\documentclass[reqno,12pt]{amsart}
\usepackage{a4wide}
\usepackage{amsmath}
\usepackage{amssymb}
\usepackage{amsthm}
\usepackage{epsfig,graphics,mathrsfs,graphicx,color}

\numberwithin{equation}{section}

\def\accentsfrancais{applemac}
   \RequirePackage[\accentsfrancais]{inputenc}

\newtheorem{theo}{Theorem}

\newtheorem{coro}{Corollary}
\newtheorem{prop}{Proposition}

\newtheorem{lem}{Lemma}

\theoremstyle{remark}
\newtheorem*{Remark}{Remark}

\def\zu{[0,1]}
\def\R{\mathbb{R}}

\def\rr{\rho}
\def\rhob{\overline{\rho}}

\def\al{\alpha}

\def\ep{\varepsilon}

\def\si{\sigma}

\def\({\left(}
\def\){\right)}
\def\[{\left[}
\def\]{\right]}

\def\dd{\textup{d}}

\def\si{\sigma}

\title[]{ Hardy-Littlewood series and\\ even continued fractions}

\author[]{Tanguy Rivoal and St\'ephane Seuret}
\date{\today}

\address{Tanguy Rivoal, Institut Fourier, CNRS et Universit\'e Grenoble 1, 
100 rue des maths, BP 74, 38402 St Martin d'H\`eres cedex, France.
}

\address{St\'ephane Seuret,
LAMA, UMR CNRS 8050,
Universit\'e Paris-Est, 
61 avenue du G\'en\'eral de Gaulle, 94010 Cr\'eteil Cedex, France. 
}



\begin{document}

\begin{abstract}
For any $s\in (1/2,1]$, the series $F_s(x)=\sum_{n=1}^{\infty} e^{i\pi n^2 x}/n^s$ converges 
almost everywhere on $[-1,1]$ by 
a result of Hardy-Littlewood concerning the growth of the sums 
$\sum_{n=1}^{N} e^{i\pi n^2 x}$, but not everywhere. However, there does not 
yet exist an intrinsic description 
of the set of convergence for $F_s$. 
In this paper, we define in terms of even continued fractions 
a subset of points of $[-1,1]$ of full measure where the series converges. 

As an intermediate step, we prove that, for $s>0$, the 
sequence of functions 
$$
\sum_{n=1}^{N} \frac{e^{i\pi n^2 x}} {n^s}
-e^{\textup{sign}(x)i\frac{\pi}4} \vert x\vert^{s-\frac12}
\sum_{n=1}^{\lfloor N \vert x\vert \rfloor} \frac{e^{-i\pi n^2/x}}{n^s}
$$ 
converges when $N\to \infty$ to a  
function $\Omega_s$ continuous 
on $[-1,1]\setminus\{0\}$ with (at most) a singularity at $x=0$ of type 
$x^{\frac{s-1}{2}}$ $(s\neq 1)$ or a logarithmic singularity ($s=1$). 
We provide an explicit expression for~$\Omega_s$ and the error term. 

Finally, we study thoroughly the convergence properties of certain series 
defined in term of the convergents of the even continued fraction of an irrational number.
\end{abstract}

\maketitle

\section{Introduction}\label{sec:0}

The famous lacunary Fourier series
\begin{equation}
\label{riemann}
\sum_{k=1}^{\infty} \frac{\sin(\pi k^2 x)}{k^2} 
\end{equation}
was proposed by Riemann in the 1850's as an example of continuous but nowhere 
differentiable function. Since then, this 
series has drawn much attention from many mathematicians 
(amongst them, Hardy and  Littlewood), 
and its complete local study was finally achieved by Gerver in \cite {Gerver} and Jaffard 
in \cite{JaffRiemann}. In particular, its local regularity at 
a point $x$ depends on the Diophantine type of $x$, and it is 
differentiable only at rationals $p/q$ where $p$ and $q$ are both odd. 

In this article, we study the series defined for  $(x,t)\in \R^2$ and $s\in \R^+$ by 
$$
F_s(x,t)=\sum_{k=1}^{\infty} \frac{e^{i\pi k^2 x+2i\pi k t}}{k^s}
$$
and we denote by $F_{s,n}(x,t)=\sum_{k=1}^{n} \frac{e^{i\pi k^2 x+2i\pi k t}}{k^s}$ its $n$-th partial sum. 
Both are periodic functions of period $2$ in $x$ and $1$ in $t$. 
For $s=2$ and $t=0$ the imaginary part of is~\eqref{riemann}.
For any fixed $t$, if $s>1/2$, $F_s$ is in $L^2(-1,1)$ and it converges almost everywhere by 
Carleson's theorem. It is not everywhere convergent however.
One of the aim of this paper is to understand better the convergence 
of $F_s(x,t)$ especially when $t=0$. 

\medskip

We set $\rho=\exp(i\pi/4)$, $\si(x)=1$, resp. $-1$ if $x>0$, resp. $x<0$, and $\sigma(0)=0$.
We define $\log(z)=\ln\vert z\vert+ i\arg(z)$ with $-\pi < \arg(z)\le \pi$. We denote by $\lfloor x \rfloor$ 
and $\{x\}$ the integer part and fractional part respectively of a real number $x$. 
For $x>0$,  $t\in \mathbb R$ and $s\ge 0$, we set 
\begin{multline}\label{eq:8}
I_s(x,t) = \int\limits_{1/2-\rho \infty}^{1/2+\rho \infty} \frac{e^{i\pi z^2 x+2i\pi z \{t\}}}
{z^s(1-e^{2i\pi z})} \dd z
\\
+\rho x^s\int\limits_{-\infty}^{\infty} e^{-\pi x u^2}
\left(\sum_{k=1}^{\infty} e^{-i\pi (k-\{t\})^2/x} 
\bigg(\frac{1}{(\rho xu+k-\{t\})^s}-\frac{1}{k^s}\bigg) \right) \dd u.
\end{multline}
This function is well-defined and if $s=0$, the second 
integral is equal to $0$ (because the series vanishes). 
We then define a function $\Omega_s(x,t)$ as follows: 
$$ \Omega_s(x,t)= \begin{cases} I_s(x,t) & \mbox{ when } x>0,
\\ 
\\
\overline{I_s(-x,-t)} & \mbox{ when } x<0.\end{cases}$$
For simplicity, given a function $f(x,t)$, we will write $f(x)$ for $f(x,0)$. The function $\Omega_s(x,t)$ will be 
particularly important in this paper.

\subsection{Statement of our main results}

Our first result is a   consequence of the celebrated 
``approximate functional equation for the theta series''  
of Hardy and Littlewood (Proposition~\ref{prop:H-L} 
in Section~\ref{sec:HL}), which corresponds exactly to the case $s=0$ 
in our Theorem~\ref{theo:1} below;  
see~\cite{choque} where many references and historical notes are given.

\begin{theo}\label{theo:4}
Let $x$ be an irrational number in $(0,1)$ whose (regular) continued 
fraction is denoted by $(P_k/Q_k)_{k\ge 0}$,  and let  $t\in \mathbb R$.

$(i)$ If $s\in(\frac12,1)$ and 
\begin{equation}\label{eq:cvQps1}
\sum_{k=0}^{\infty} \frac{Q_{k+1}^{\frac{1-s}2}}{Q_k^{s/2}}<\infty,
\end{equation}
then $F_s(x,t)$ is absolutely convergent.

$(ii)$ If $s=1$ and 
\begin{equation}\label{eq:cvQps2}
\sum_{k= 0}^{\infty} \frac{\log(Q_{k+1})}{\sqrt{Q_k}}<\infty,
\end{equation}
then $F_1(x,t)$ is absolutely convergent.
\end{theo}

Conditions~\eqref{eq:cvQps1} and~\eqref{eq:cvQps2} hold 
for Lebesgue-almost all $x$.   It is possible to prove a quantitative version of Therorem~\ref{theo:4}. 
We need to introduce more notations.
We introduce the two transformations  
$$
G(x)= \left\{\frac1x \right\}, \qquad
\widetilde{G}(x,t)=\left\{\frac12\left[\frac 1x\right]-\frac tx \right\}.
$$
Then, for all $x$ satisfying at least the 
conditions~\eqref{eq:cvQps1} or \eqref{eq:cvQps2}, it can be 
proved~(\footnote{The details will not be given here because the process is similar to that leading to  Theorem~\ref{theo:2} and this 
would add nearly ten more pages to the paper.}) that   
\begin{equation}\label{eq:GtildeG}
F_s(x,t) = \sum_{j=0}^{\infty} 
e^{i\frac{\pi}8(1+(-1)^{j-1})+i\pi\sum\limits_{1\le \ell \le j}^{} (-1)^{\ell} \widehat{G}_{\ell}(x,t)}
(xG(x)\cdots G^{j-1}(x))^{s-\frac12} 
\Omega_{j,s}(G^j(x),\widetilde{G}_j(x,t)),
\end{equation}
where $\Omega_{j,s}(x,t)=\Omega_s(x,t)$ if $j$ is even, 
$\Omega_{j,s}(x,t)=\overline{\Omega_s(x,t)}$ if $j$ is odd, and 
\begin{align*}
\widetilde{G}_0(x,t) = x, \quad \widehat{G}_{0}(x,t)=t,& \quad 
\widetilde{G}_1(x,t) = \widetilde{G}(x,t),  
\quad \widehat{G}_{1}(x,t)=\frac{t^2}x
\\
\widetilde{G}_{j+1}(x,t) = \widetilde{G}_1(T^j(x),\widetilde{G}_j(x,t)),& \qquad 
\widehat{G}_{j+1}(x,t) = \widehat{G}_{1}(T^j(x),\widetilde{G}_j(x,t))\qquad (j\ge 0).
\end{align*}
Eq.~\eqref{eq:GtildeG} holds very generally, the right-hand side converges quickly and the appearence of Gauss' transform $G$ is a nice feature. But this is at the cost of the simultaneous appearance of the operator $\widetilde{G}$ and this makes~\eqref{eq:GtildeG} looks very complicated, even when $t=0$ because $\widetilde{G}_j(x,0)\neq 0$ in general.

\medskip

However, the underlying modular nature of $F_s(x,t)$ implies that the transformation of $[-1,1]\setminus\{0\}$ given by $$T(x)=-\frac1x \mod 2$$
is  more natural than Gauss' in this specific study, and in particular it leads to another expression (i.e,~\eqref{eq:TtildeT} below) for $F_s(x,t)$ which is formally similar to~\eqref{eq:GtildeG} but simpler. The comparison of both approaches is one of our motivations. 

Our next theorem below explains what we mean by ``the modular nature of $F_s(x,t)$'' and the subsequent theorems are devoted to convergence conditions of $F_s(x,t)$ (mainly when $t=0$) in terms of series defined by the operator $T$, as well as their relations with Theorem~\ref{theo:4}.

\begin{theo} \label{theo:1}

$(i)$ For any $x\in [-1,1]\setminus \{0\}$, $t\in \mathbb R$, $s\ge 0$, we have the estimate
\begin{multline}\label{eq:1}
F_{s,n}(x,t) - e^{\si(x) i\frac{\pi}4} e^{-i\pi \frac{\{ \sigma(x) t\}^2}x}
\vert x\vert^{s-\frac12} F_{s,\lfloor n \vert x\vert \rfloor}\Big(-\frac1x,\frac{\{\si(x) t\}}x\Big)\\
= \Omega_s(x,\si(x) t) +\mathcal{O}\left( \frac{1}{n^s \sqrt{x}}\right).
\end{multline}
 when $n$ tends to infinity. The implicit constant depends on $s$ and $t$, but not on $x$.
 
$(ii)$ When $0\le s\leq 1$, the function $\Omega_s(x)$ is 
continuous on $\mathbb R\setminus\{0\}$, differentiable at any rational 
number $p/q$ with $p,q$ both odd, 
and 
$$
\Omega_s(x)-\frac{\rho^{1-s}\Gamma(\frac{1-s}{2})}{2\pi^{\frac{1-s}{2}}}
\vert x \vert^{\frac{s-1}2} \quad ( 0 \le s < 1)
\quad \textup{and} \quad
\Omega_1(x)-\log (1/\sqrt{\vert x \vert})
$$ 
are bounded on $\mathbb R$.

$(iii)$ When $s>1$, the function $\Omega_s(x)$ is differentiable on $\R\setminus\{0\}$ and continuous at $0$.
\end{theo}
\begin{Remark} The function  $\Omega_0(x)$  
is the same as the one used by Cellarosi~\cite{CELLA1} and 
Fedotov-Klopp~\cite{FEDKLOPP}. When $s>1$,  the proof of item $(ii)$ yields only that $\Omega_s$ is bounded around $0$. 

See Figures~\ref{fig3} and~\ref{fig4} for an illustration of Theorem~\ref{theo:1}.
\end{Remark}

As $n\to +\infty$, the left hand side of~\eqref{eq:1} tends 
to $\Omega_s(x,t)$ when $x>0$. The resulting ``modular'' equation
\begin{equation}\label{eq:2}
F_s(x,t)-e^{i\frac{\pi}4} e^{-i\pi \frac{t^2}x}
x^{s-\frac12}\,F_s\Big(-\frac 1x,\frac{t}x \Big)
= \Omega_s(x,t)
\end{equation}
holds a priori at least {\em almost everywhere}  for $x\in (0,1)$ for 
any fixed $s\geq 0 $ and $t\in [0,1)$, 
and Theorem~\ref{theo:1} shows in which sense 
we can say it holds {\em everywhere}. For other examples of this 
phenomenon, see~\cite{conrey, rivroq} for instance. Of course, if $s>1$,~\eqref{eq:2} 
holds for all $x$.

\begin{figure} 
\centering
\epsfig{figure=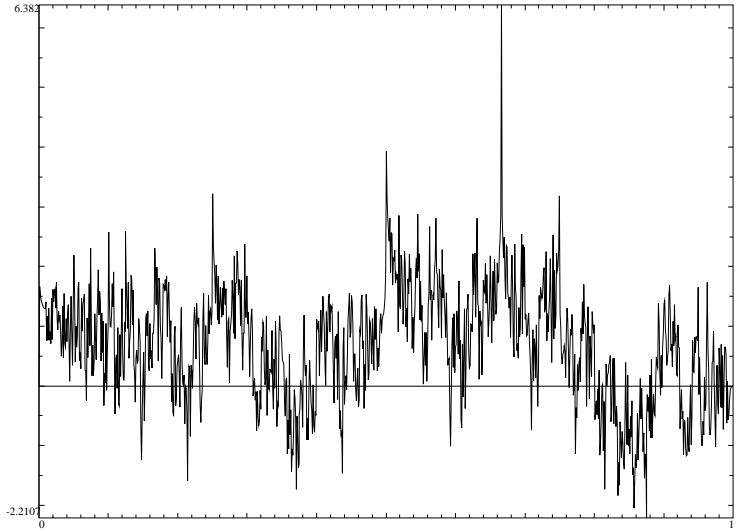,angle=-90, width=0.6\textwidth}
\caption{Plot of $\textup{Im}(F_{0.7,100}(x))$ on $[0,2]$}
\label{fig3} 
\end{figure}

\begin{figure} 
\centering
\epsfig{figure=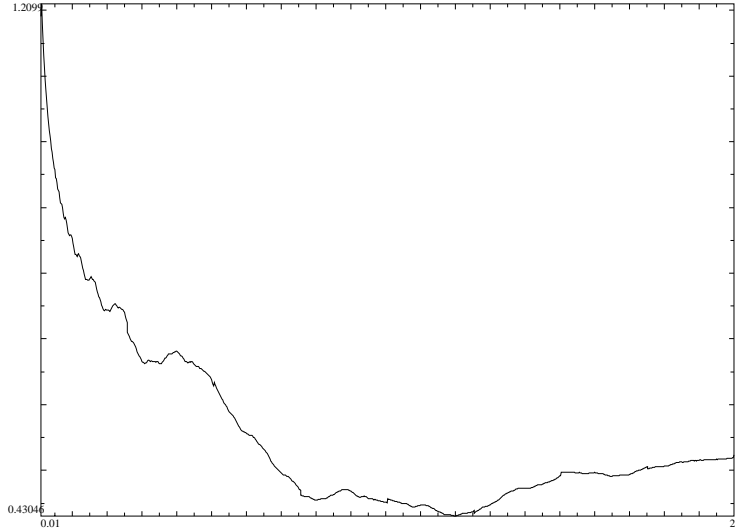,angle=-90, width=0.6\textwidth}
\caption{Plot of $\textup{Im}(F_{0.7,1000}(x)-e^{i\pi/4}x^{0.2}F_{0.7,\lfloor 1000x\rfloor}(-1/x))$ on $[0,2]$}
\label{fig4} 
\end{figure}

In fact, we will obtain a more precise estimate for the error term in \eqref{eq:1}, 
uniform in $x \in[-1,1]\setminus\{0\}$, $n\ge 0$ and $t\in \mathbb R$: 
\begin{multline}\label{eq:1bis}
\mathcal{O}\left( \frac{1}{n^s \sqrt{x}}\right)  = 
\mathcal{O}\left(\frac{\vert x\vert^{s-\frac12}}{\big(\lfloor n\vert x\vert+\si(x) 
t\rfloor+1- \si(x) t\big)^s}\right)
\\
+ 
\begin{cases}
\mathcal{O} \left( \min\Big(\frac{1}{(n+1)^s\sqrt{\vert x\vert}}, 
\frac{1}{\vert x\vert^{\frac{1-s}{2}}}\Big) \right) \quad (s\neq 1)
\\
\mathcal{O} \left( \min\Big(\frac{1}{(n+1)\sqrt{\vert x\vert}}, 1+\big\vert \log\big((n+1) 
\sqrt{\vert x\vert}\big)\big\vert\Big)\right) \quad (s=1),
\end{cases}
\end{multline}
where the constants in the $\mathcal{O}$ on the right-hand side 
depend now at most on $s$ and are effective. 
We need this refinement to prove Theorem~\ref{theo:2} below.
  
Theorem~\ref{theo:1} will be used to get informations of the convergence of $F_s(x)$ in terms 
of the diophantine properties of $x$. In the sequel $T^m(x)$ 
denotes the $m$-th iterate of $x$ by $T$. By 
$2$-periodicity of $T$, Eq.~\eqref{eq:1} can be rewritten as follows (when $t=0$): for any 
$x\in [-1,1]\setminus \{0\}$, $s>\frac12$ and integer $n\ge 0$, 
\begin{multline}\label{eq:11}
F_{s,n}(x)=e^{\si(x)i\frac{\pi}4} 
\vert x\vert^{s-\frac12}F_{s,\lfloor n \vert x\vert \rfloor} \big(T(x)\big)
+\Omega_s(x)
\\
+ \begin{cases}
\mathcal{O} \left( \min\Big(\frac{1}{(n+1)^s\sqrt{\vert x\vert}}, 
\, \frac{1}{\vert x\vert^{\frac{1-s}{2}}}\Big) \right) \qquad (s\neq 1)
\\
\mathcal{O} \left( \min\Big(\frac{1}{(n+1)\sqrt{\vert x\vert}}, 
\, 1+\big\vert \log\big((n+1) \sqrt{\vert x\vert}\big)\big\vert\Big)\right) \qquad (s=1).
\end{cases}
\end{multline}
(When $t=0$, 
the error term $\vert x\vert^{s-\frac12}\cdot \big(\lfloor n\vert x\vert+\si(x) 
t\rfloor+1- \si(x) t\big)^{-s}$ in~\eqref{eq:1bis} is absorbed by the error term in~\eqref{eq:11}, 
for some constant that depends only on $s$;  this is enough for the application we have in mind.) 
The second sum $F_{s,\lfloor n \vert x\vert \rfloor} \big(T(x)\big)$  in~\eqref{eq:11} involves less terms than the first one for any irrational number 
in $(-1,1)$, because $\lfloor n \vert x\vert \rfloor<n$. Hence for any fixed $n$ and $x$, 
we can iterate~\eqref{eq:11} because $T(x)\in (-1,1)$. 
After a finite number of steps (say $\ell$, which depends on $x$), we get an empty 
sum $F_{s,0}(T^\ell(x))=0$ 
on the right hand side together with a finite sum  
defined in terms of iterates of $\Omega_s(x)$ 
and a quantity we expect to be an error term (i.e., that tends 
to $0$ as $n$ tends to infinity under suitable condition on~$x$).

We prove the following result.
\begin{theo}\label{theo:2} 
Let $x\in (-1,1)$ be an irrational number. 
\\
\indent $(i)$
 If $s\in(\frac12 ,1)$ and if 
\begin{equation}\label{eq:riv10}
\sum_{j=0}^{\infty} 
\frac{\vert xT(x)\cdots T^{j-1}(x)\vert^{s- \frac12}}
{\vert T^j(x)\vert^{\frac{1-s}{2}}}<\infty,
\end{equation}
then $F_s(x)$ is also convergent and the following identity holds:
\begin{equation}\label{eq:riv11}
F_s(x)=\sum_{j=0}^{\infty}e^{i\frac{\pi}4\sum\limits_{\ell=0}^{j-1}\si(T^\ell x)}
\vert xT(x)\cdots T^{j-1}(x)\vert^{s-\frac12} \,\Omega_s\big(T^j(x)\big).
\end{equation}

$(ii)$ If 
\begin{eqnarray}\label{eq:33}
&&\sum_{j=0}^{\infty} \sqrt{\vert xT(x)\cdots T^{j-1}(x)\vert}<\infty\\
\mbox{ and }&& \label{eq:3}
\sum_{j=0}^{\infty} \sqrt{\vert xT(x)\cdots T^{j-1}(x)\vert} 
\,\log \Big(\frac1{\vert T^j x\vert}\Big)<\infty, 
\end{eqnarray}
then $F_1(x)$ is also convergent and the following identity holds:
\begin{equation}\label{eq:7}
F_1(x)=\sum_{j=0}^{\infty}e^{i\frac{\pi}4\sum\limits_{\ell=0}^{j-1}\si(T^\ell x)}
\sqrt{\vert xT(x)\cdots T^{j-1}(x)\vert} \,\Omega_1\big(T^j(x)\big).
\end{equation}
\end{theo}

The series in~\eqref{eq:riv10},~\eqref{eq:33} and~\eqref{eq:3} do not converge 
everywhere. 
It was an open question whether such series converge 
Lebesgue-almost everywhere. Note that the results in the cited papers of 
Cellarosi~\cite{CELLA1}, Kraaikamp-Lopes~\cite{K-L},  Schweiger~\cite{S1, S2}, and Sinai~\cite{Sinai} give estimates for 
the average behavior of $\vert xT(x)\cdots T^{j-1}(x)\vert$ 
when $j$ tends to infinity, but these estimates are not 
sharp enough~(\footnote{It is known that 
$T$ is ergodic with respect to a measure $\nu$ supported on $[-1,1]$ but, in 
contrast with the ergodic theory of 
Gauss' transformation $G$, the measure $\nu$   is infinite. 
As a consequence, the analogue of Birkoff's ergodic theorem is not known and one must 
content with  ``convergence in probability'' results (see~\cite{CELLA1}), which are 
not well suited to our study.}) to guarantee the almost everywhere convergence. 
It would be interesting 
to know if $F_1(x)$ converges under the assumption of convergence of~\eqref{eq:3} only.

\subsection{More results around Theorem~\ref{theo:2}}

We now precise the diophantine content of Theorem~\ref{theo:2}.

\begin{theo}\label{theo:3} Let $\alpha>0$, $\beta \geq 0$, 
and set $\beta_\alpha=\displaystyle \frac{\sqrt{\alpha^2+4}-1}{2}$. 

$(i)$   If  $0\leq \beta<\beta_\alpha$, then the series 
\begin{equation}
\label{eq:seuret444}
\sum_{j=0}^{\infty} \frac
 {\vert xT(x)\cdots T^{j-1}(x)\vert ^\alpha }{|T^j(x)|^\beta}
\end{equation}
converges if
\begin{equation}
\label{condition1bis} \sum_{n=1}^{\infty}      
\frac{Q_{n+1}^{\beta+1 }}{ Q_{n }^{\alpha+\beta+1}}<\infty.
\end{equation}

$(ii)$  If  $\beta>\beta_\alpha$, then the series~\eqref{eq:seuret444} converges if
\begin{equation}
\label{condition1bisbis} \sum_{n=1}^{\infty}   
\frac{  Q_{n+2}^{\beta}} { Q_{n}^{\alpha+\beta}} 
<\infty
\end{equation}

$(iii)$  If $\beta=\beta_\alpha$, then~\eqref{condition1bis} 
and~\eqref{condition1bisbis} impose the convergence of~\eqref{eq:seuret444}.

$(iv)$ If $\alpha\geq 0$ and 
\begin{equation}
\label{condlog}
\sum_{n=1}^{\infty} \left(\frac{\log(Q_{n+1})}{Q_{n-1}^\alpha} + 
\frac{ Q_{n+1} \log(Q_{n+1})^2}{Q_n^{1+\alpha}} \right)<\infty,
\end{equation}
then the series 
$$
\sum_{j=0}^{\infty}
 \vert xT(x)\cdots T^{j-1}(x)\vert^\alpha \log \left(\frac{1}{T^j(x)} \right)
$$
converges.
 \end{theo}

The condition \eqref{condition1bis} can be simplified according to the 
values of $\alpha$ and $\beta$, in terms of the irrationality exponent 
$\mu(x)$ of an irrational  $x\in \R$, defined as  
$$
\mu(x)=\sup\left\{\mu \geq 1: \  \left|x-\frac{p}{q}\right|<\frac{1}{q^{\mu}} 
\ \mbox{ for infinitely many integers $q\geq 1$} \right\}. 
$$
It is well-known that $\mu(x)=2$ for almost every real numbers $x$.
When $s=1$,  choosing $\alpha=\frac12$, $\beta=0$, one sees that~\eqref{condlog} 
implies~\eqref{condition1bis}.   
When $s\in (\frac12,1)$, putting  $\alpha=s-\frac12$ and $\beta=\frac{1-s}2$, a simple 
computation shows  that   $\beta<\beta_\alpha$  in Theorem~\ref{theo:3}.

\begin{coro}\label{coro:4} 
$(i)$ 
If $1/2<s<1$ and 
\begin{equation}\label{eq:riv2806}
\sum_{n=1}^{\infty} \frac{Q_{n+1}^{\frac{3-s}{2}}}{Q_{n }^{1+\frac s2}}
\end{equation}
is convergent,  
then the identity~\eqref{eq:riv11}  holds true. The series~\eqref{eq:riv2806} converges 
for every $x$ such that $\mu(x)  <1+\frac{2+s}{3-s}$.

$(ii)$  If  
\begin{equation}\label{eq:riv2807}
\sum_{n=1}^{\infty} \left( \frac{\log(Q_{n+1})}{\sqrt{Q_{n-1}}} 
+ \frac{ Q_{n+1} \log(Q_{n+1})^2}{Q_n^{3/2}}\right)
\end{equation}
converges, then the equality \eqref{eq:7} holds true. The series~\eqref{eq:riv2807} converges for 
every $x$ such that $\mu(x)  < \frac52$.
\end{coro}
Corollary~\ref{coro:4} is not entirely satisfying, since 
the series  $F_s(x)$ converges (even absolutely) when 
a weaker condition on the standard convergents of $x$ holds 
(Theorem \ref{theo:4}, conditions \eqref{eq:cvQps1} and \eqref{eq:cvQps2}). The main reason 
for this discrepency is the factor $\exp(i\frac{\pi}4\sum_{0\le \ell <j}\si(T^\ell x))$: it is 
present in~\eqref{eq:riv11} but not in~\eqref{eq:riv10}, respectively in~\eqref{eq:7} but not 
in~\eqref{eq:33}-\eqref{eq:3}. Our next result shows that the role of 
this factor is very important, even though its modulus is $1$. We explain after 
the theorem why we are not 
able to keep track of it in our proof of Theorem~\ref{theo:2}.
\begin{theo}
\label{theo:5} 
$(i)$ Let $\Omega$ be a bounded function, differentiable at $x=1$ and 
$x=-1$ (in particular, if $\Omega\equiv 1$). Then for any $\alpha>0$ and any irrational number 
$x\in (0,1)$, the series
\begin{equation}
\label{seriesfinal}
  \sum_{j=1}^{\infty}e^{i\frac{\pi}4\sum\limits_{\ell=0}^{j-1}\si(T^\ell x)}
\vert xT(x)\cdots T^{j-1}(x)\vert^{\alpha} \,\Omega \big(T^j(x)\big)
\end{equation}
converges.

$(ii)$ For any $\al>0$, any $\beta \in \mathbb R$ and any irrational number 
$x\in (0,1)$, the series
\begin{equation}
\label{eq:seuret4}
\sum_{j=0}^{\infty}e^{i\frac{\pi}4\sum\limits_{\ell=0}^{j-1}\si(T^\ell x)} \frac
 {\vert xT(x)\cdots T^{j-1}(x)\vert ^\alpha }{|T^j(x)|^\beta}
 \end{equation}
converges if
\begin{equation}
\label{eq:seuretcond4}
\sum_{n=1}^{\infty} \frac{ Q_{n+1}^\beta}{Q_{n}^{\alpha+\beta}}<\infty.
\end{equation}

$(iii)$ For any $\al>0$ and any irrational number 
$x\in (0,1)$, 
the series
\begin{equation}
\label{eq:seuret3}
\sum_{j=0}^{\infty}e^{i\frac{\pi}4\sum\limits_{\ell=0}^{j-1}\si(T^\ell x)}
 \vert xT(x)\cdots T^{j-1}(x)\vert^\alpha \log \left(\frac{1}{T^j(x)} \right)
 \end{equation}
 converges if
\begin{equation}
\label{eq:seuretcond3}
\sum_{n=1}^{\infty} \frac{\log (Q_{n+1})}{Q_{n}^\alpha} <\infty.
 \end{equation}
 \end{theo}
These convergence properties are essentially optimal. They are 
very different from the absolute convergence properties (which are also essentially optimal), 
as stated in Theorem~\ref{theo:2}. 

In fact, in our proof of Theorem~\ref{theo:2}, only one technical detail impedes us to 
prove that formulas~\eqref{eq:riv11} 
and~\eqref{eq:7}  hold true when conditions~\eqref{eq:cvQps1} 
and~\eqref{eq:cvQps2} are satisfied (and not only when the more 
constraining  conditions in $(i)$ and $(ii)$ of Corollary~\ref{coro:4} hold). 
Indeed, we show that  the convergence of the series~\eqref{eq:riv11}  
and~\eqref{eq:7} is equivalent to the convergence of three auxiliary 
simpler series (see equations~\eqref{eq:161} and~\eqref{eq:16}). 
For two of these series,  their convergence 
follows from Theorem~\ref{theo:5}, and the convergence conditions 
are optimal (i.e.,  when conditions~\eqref{eq:cvQps1} 
and~\eqref{eq:cvQps2} are satisfied). For the third series, which 
contains heuristically a sort of "error" term, we do not have an  
estimate precise enough to apply Theorem~\ref{theo:5}, and we can 
only use Theorem~\ref{theo:3} and the conditions ensuring absolute 
convergence of the series~\eqref{eq:seuret3} and~\eqref{eq:seuret4}, which are stronger. 
Nevertheless, we do believe that conditions~\eqref{eq:cvQps1} 
and~\eqref{eq:cvQps2} imply 
the identities~\eqref{eq:riv11} and~\eqref{eq:7} respectively.

\subsection{Some further remarks}

For $t$ not necessarily an integer, 
the iteration of~\eqref{eq:1} leads to an identity, for which we 
have to introduce the following sequences of operators:  
\begin{align*}
\widetilde{T}_0(x,t) = x, \quad \widehat{T}_{0}(x,t)=t &, \quad 
\widetilde{T}_1(x,t) = \left\{\frac{\{\sigma(x)t\}}{x}\right\}, 
\quad \widehat{T}_{1}(x,t)=\frac{\{\sigma(x)t\}^2}x
\\
\widetilde{T}_{j+1}(x,t) = \widetilde{T}_1(T^j(x),\widetilde{T}_j(x,t)),& \qquad 
\widehat{T}_{j+1}(x,t) = \widehat{T}_{1}(T^j(x),\widetilde{T}_j(x,t))\qquad (j\ge 0).
\end{align*}
Then,  
for any fixed $t\in [0,1]$ and $s>\frac12$, 
the following identity~(\footnote{This could be precisely described as in 
Theorem~\ref{theo:2} but, again, we skip the details to shorten the length of the paper.}) holds for almost every $x$: 
\begin{equation}\label{eq:TtildeT}
F_s(x,t)=\sum_{j=0}^{\infty} e^{i\pi \sum\limits_{\ell=0}^{j-1} 
\left(\frac{1}4\si(T^\ell x)-\widehat{T}_{\ell+1}(x,t) \right)}
\vert xT(x)\cdots T^{j-1}(x)\vert^{s-\frac12} \,\Omega_s\big(T^j(x), \widetilde{T}_j(x,t)\big).
\end{equation}
This new representation of $F_s(x,t)$ is similar to Identity~\eqref{eq:GtildeG} displayed after 
Theorem~\ref{theo:4}. For $x$ unspecified, Eq.~\eqref{eq:TtildeT} is simpler than~\eqref{eq:GtildeG} when $t=0$,  because $\widetilde{T}_j(x,0)=\widehat{T}_j(x,0)=0$ for all $j$ while neither 
$\widetilde{G}_j(x,0)$ nor $\widehat{T}_j(x,0)$ necessarily vanish. Finally, it 
is easy to see that when $x$ has {\em even} partial quotients and $t=0$, then the summands 
of~\eqref{eq:TtildeT} and~\eqref{eq:GtildeG} are equal. The simplicity of \eqref{eq:TtildeT} (relatively 
to~\eqref{eq:GtildeG}) when $t=0$ was our motivation to make the detailed study of various series defined in term of the operator $T$, for which apparently nothing was done in the direction of our results.

\medskip

Finally, when $s>1$, the series $F_s(x,t)$ obviously  
converges absolutely for any real numbers $x$ and~$t$. It turns out that Identities~\eqref{eq:GtildeG} and~\eqref{eq:TtildeT} hold for any 
$t$ and any irrational number $x$, and with minor modification for any rational number $x$ as well. On the one hand, this is not difficult to prove for~\eqref{eq:GtildeG}, whose right hand side converges very quickly. 
On the other hand, the convergence of the right-hand side of~\eqref{eq:TtildeT} for all irrational $x$ is a consequence of Theorem~\ref{theo:5}$(i)$ applied with $\alpha=s-\frac12$ because for $s>1$, $\Omega_s(x)$ is bounded on $[-1,1]$ and differentiable at $x=\pm 1$ by Theorem~\ref{theo:1}$(iii)$.

We note that $T$ is closely related to the Theta group, a subgroup of $SL_2(\mathbb Z)$ of the matrices $\binom{a\;\; b}{c\;\; d}$ with $a\equiv d \mod 2$, $b\equiv c \mod 2$; see~\cite{K-L}. This relation has been used in the papers~\cite{duisd, itatsu} to study the (non)-derivability of Riemann series $\textup{Im}(F_2(x))$, culminating with Jaffard's determination of its spectrum of singularities~\cite{JaffRiemann}. It would be very interesting to know if Jaffard's results can be recovered by a direct study of~\eqref{eq:TtildeT} in the case $s=2$ and $t=0$, which reads
$$
F_2(x)= \sum_{k=1}^{\infty} \frac{e^{i \pi k^2 x}}{k^2} = \sum_{j=0}^{\infty}e^{i\frac{\pi}4\sum\limits_{\ell=0}^{j-1}\si(T^\ell x)}
\vert xT(x)\cdots T^{j-1}(x)\vert^{\frac 32} \,\Omega_2\big(T^j(x)\big),  
$$
where $\Omega_2(x)$ is differentiable on $[-1,1]\setminus\{0\}$, and continuous at 0.

\begin{figure} 
\centering
\epsfig{figure=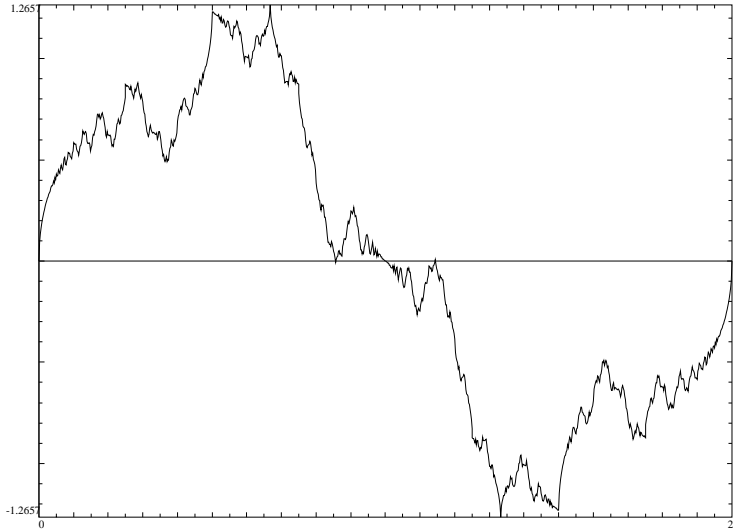, angle=-90, width=0.6\textwidth}
\caption{The ``non-differentiable'' Riemann function  $\textup{Im}(F_2(x))$ on $[0,2]$}
\label{fig0} 
\end{figure}

\bigskip

The paper is organized as follows.  We start by recalling some 
facts on regular and even continued fractions in Section~\ref{sec:1}. Then, 
in Section~\ref{sec:HL}, we prove Theorem~\ref{theo:4}.  Section~\ref{sec:newthm}, 
which is rather long, contains the proof of the approximate functional  
equation for $F_{s,n}$ (part $(i)$ of  Theorem~\ref{theo:1}). The 
second part $(ii)$ of Theorem~\ref{theo:1}, i.e. the regularity 
properties of the functions $\Omega_s$, is dealt with in 
Section~\ref{sec:2}. Theorem~\ref{theo:2} and the Diophantine 
identities~\eqref{eq:riv11} 
and~\eqref{eq:7} are proven in Section~\ref{sec:4}. Finally, 
the  standard and absolute convergence properties of the 
series~\eqref{eq:seuret3} and~\eqref{eq:seuret4}  (Theorems~\ref{theo:5}    
and~\ref{theo:3}) are studied in Sections~\ref{sec:10_2} and~\ref{sec:10_1}. 

\section{Basic properties of regular and even continued fractions}\label{sec:1}

 The regular  theory of continued fractions is well-known, and is related to 
Gauss dynamical system $G:x\in (0,1]\mapsto 1/x \mod 1$.  In the following, we 
use capital letters when we refer to the regular convergent $P_n/Q_n$ of an irrational 
number $x\in \R$. We set
\begin{equation}
\label{eq:SCF}
\frac{P_n}{Q_n}:=\lfloor x \rfloor+ \frac{1}{A_1+\dfrac{1}{A_2+\dfrac{1}
{\ddots +\dfrac{1}{A_n}}}}
\  \ \mbox{ 
with
} \ \ 
A_n=  \left\lfloor \frac{1}{G^{n}( x)} \right \rfloor,
\end{equation}
where $G^n$ is the $n$-th iterate of $G$. We write the RCF (Regular Continued Fraction) of $x$ as
$$
x=\lfloor x \rfloor+[A_1,A_2,....]_R.
$$
Recall that one has the recurrence relations 
$$
P_{n+1}=A_{n+1}P_n+P_{n-1}  \ \mbox { and } \ Q_{n+1}=A_{n+1}Q_n+Q_{n-1}.
$$

It is also classical that in this case, for every irrational $x\in \zu$, 
\begin{equation}
\label{xGxG2x}
\left|xG(x)G^2(x) \cdots G^n(x)\right| \leq  |Q_nx-P_n| \leq (Q_{n+1})^{-1}.
\end{equation}
This guarantees  the convergence for every $\alpha>0$ and $x\in \R$ of the series 
\begin{equation}
\label{convergesum1}
\sum_{n\geq 1} \left|xG(x)G^2(x) \cdots G^n(x)\right|^\alpha.
\end{equation}

 \medskip
 
  In Theorems \ref{theo:2} and \ref{theo:3}, the natural 
underlying dynamical system is the one generated by the map 
$T: [-1,1]\setminus \{0\} \mapsto  [-1,1] $ defined by 
$T(x) = -1/x \mod 2$. As is classical with the Gauss map $G$, 
using this transformation $T$ one can associate with each irrational 
real number $x\in [-1,1]\setminus \{0\} $ a kind of continued 
fraction:  for every $j\geq 1$, denote by $a_j$ the unique even 
number such that $ T^j (x) -a_j \in (-1,1)$, and set 
$e_j = \si(T^j (x))$. Then one has the unique decomposition called
 the even continued fraction (ECF) expansion (see \cite {K-L} for instance)
\begin{equation}
\label{eq:111}
 x=\dfrac{e_1}{a_1+\dfrac{e_2}{a_2+\dfrac{e_3}{a_3+ ...}}}.
 \end{equation}
 The difference with the RCF expansion \eqref{eq:SCF} is twofold: first, 
only even integers $(a_j)_{j\geq 1}$ are allowed in the decomposition, and 
second the integers $e_1, e_2$, etc,  may take the values $1$ and $-1$ 
(not only $1$). For compactness, we write for $x\in (-1.1)$
 \begin{equation}
\label{defevenx}
x= [(e_1,a_1),(e_2,a_2),....]_E.
\end{equation}
 
Now, given the ECF \eqref{eq:111},  
we define the $n$-th convergent and the $n$-th remainder respectively as 
$$
\frac{p_n}{q_n}:=\frac{1}{a_1+\dfrac{e_1}{a_2+\dfrac{e_2}
{\ddots +\dfrac{e_{n-1}}{a_n}}}}
\  \ \mbox{ 
and 
} \ \ 
x_n:=\frac{e_n}{a_{n+1}+\dfrac{e_{n+1}}{a_{n+2}+\dfrac{e_{n+2}}{\ddots}}}.
$$
We use small letters for ECF, and capital letters for RCF.

 \medskip
 
 ECF expansions  are obtained from the RCF expansions via the following iterative 
method. Observe that for any positive integers $(A_n,A_{n+1}, A_{n+2})$ and any 
positive real number $\gamma$, one has
\begin{equation}
\label{singularization}
A_n+\dfrac{1}{A_{n+1} + \dfrac{1}{A_{n+2}+\gamma}} =(A_n+1) 
+\dfrac{-1}{2 +\dfrac{-1}{2 + .... +\dfrac{-1}{2+\dfrac{-1}{(A_{n+2}+1)+\gamma}}}},
\end{equation}
 where the term $\dfrac{-1}{2+...}$ appears exactly $A_{n+1} -1$ times. 
 
 The procedure is then as follows: Write the RCF expansion of a real 
number $x=[A_1,A_2,...]_R$ as $x=[(1,A_1),(1,A_2),...]_E$, which looks 
like an ECF expansion, except that all $e_n$ are 1 and some integers $A_n$ may be odd. 
 
 If all $A_n$ are even, then this expansion is indeed the ECF of $x$. 
 
 Otherwise,  consider the smallest index $n$ such that $A_n$ is odd, and 
apply \eqref{singularization} to transform 
 $$x=[(1,A_1), ... ,(1,A_n),(1,A_{n+1}),(1,A_{n+2}),(1,A_{n+3}),...]_E$$ into 
 $$ [(1,A_1), ..., (1,A_n+1), (-1,2),...,(-1,2), (-1,A_{n+2}+1),(1,A_{n+3}),...]_E.$$
  The odd number $A_n$ has been removed, and one iterates the procedure with this 
new ECF-like expansion, whose coefficients before $A_{n+2}+1$ are even. By 
uniqueness of the ECF expansion for irrational numbers, the  expansion 
obtained as a limit is indeed the ECF expansion of $x$.

 \smallskip
 
 From the above construction, one also derives some useful properties between 
the even and the regular convergents. Next Proposition is contained in~\cite{K-L}. 
\begin{prop}
\label{prop:ecf2}
Let $x$ be an irrational number in $(0,1)$.

$(i)$ For any $n\geq 1$, if  the regular convergent $P_n/Q_n$ is not an 
even convergent $p_j/q_j$, then necessarily $P_{n+1}/Q_{n+1}$ is  an even   convergent.

$(ii)$  If  the regular convergent $P_n/Q_n$ is equal  to  the even convergent 
$p_j/q_j$ for some $j\geq 1$, and if  $P_{n+1}/Q_{n+1}$ is  also an even convergent,  
then   $p_{j+1}/q_{j+1} = P_{n+1}/Q_{n+1}$.

$(iii)$  If  the regular convergent $P_n/Q_n$ is equal  to  the even convergent $p_j/q_j$ 
for some $j\geq 1$, and if  $P_{n+1}/Q_{n+1}$ is not an even convergent,  then  
for every $m\in\{1, ..., A_{n+2}\}$,
$$
\frac{p_{j+m}}{q_{j+m}} =  \frac{mP_{n+1}+P_{n}}{m Q_{n+1}+Q_{n}}.
$$
In particular,  $\displaystyle \frac{p_{j+A_{n+2}}}{q_{j+A_{n+2}}} = \frac{P_{n+2}}{Q_{n+2}}$.
\end{prop}
Hence,  amongst  two consecutive regular convergents, there 
is at least one even convergent, and an even convergent     $p_j/q_j$  is either 
a principal or a median convergent 
of the regular continued fraction. 
 This will be useful to prove 
Theorem \ref{theo:3}.

The next proposition gathers some useful informations about even 
continued fraction (references include~\cite[p. 307]{K-L},~\cite[p. 2027, eq. (1.13)]{Sinai},
and~\cite{S1,S2}).  
\begin{prop}
\label{prop:ecf1}
For every irrational $x\in \zu$ and every $j\geq 1$, we have
$$
q_{j+1}>q_j, \quad 
\lim_{n\to +\infty} (q_{n+1}-q_n)=+\infty
$$ 
and 
\begin{equation}
\label{majecf}
\frac{1}{2q_{j+1}}\leq \vert xT(x)\cdots T^j(x)\vert = 
\frac1{\vert q_{j+1}+e_{j+1}x_{j+1}q_j\vert}\le \frac{1}{q_{j+1}-q_j}.
\end{equation}
\end{prop}
(We will freely use these properties without necessarily quoting Proposition~\ref{prop:ecf1}.)
Unfortunately there is no uniform convergence rate for this sequence 
that guarantees the convergence of series of the form~\eqref{eq:seuret4} for every $x$.
This is in sharp contrast with the classical continued fractions and 
equation~\eqref{convergesum1}. Nevertheless we have found some optimal condition 
to guarantee the convergence of the sum $\sum_{j\ge 0}\vert xT(x)\cdots T^j(x)\vert^\alpha$, 
see Theorem~\ref{theo:3}.

\section{Proof of Theorem~\ref{theo:4}}\label{sec:HL}

In this section, we obtain sufficient conditions of convergence of $F_s(x,t)$ expressed in term of 
the usual regular continued fraction of $x$. These conditions are simple consequences of the 
following proposition, due to Hardy and Littlewood~\cite{HL}. At the end of this section, we present an 
identity which, in principle, would be a 
qualitative version of Theorem~\ref{theo:4}.

\begin{prop}\label{prop:H-L} For any irrational number $x$ in $(0,1)$ with regular continued  
fraction $(P_k/Q_k)_{k\ge 0}$ and any $t\in \mathbb R$, we have
\begin{equation}\label{eq:ineqHL}
\sum_{k=1}^N e^{i\pi k^2 x+2i\pi k t} = \mathcal{O}\left(\frac{N}{\sqrt{Q_r}}+\sqrt{Q_r}\right)
\end{equation}
for any integers $N,r\ge 0$, where the implicit constant is absolute.
\end{prop}

Eq.~\eqref{eq:ineqHL} is a corollary of the  
approximate functional equation of Hardy-Littlewood for the theta series, 
which is exactly the case $s=0$ of our Theorem~\ref{theo:1}. In this case, the 
error term reduces to $\mathcal{O}(1/\sqrt{x})$ where the constant is 
absolute. The most precise version of the functional equation is given in~\cite{cout}.
We do not reproduce the proof of~\eqref{eq:ineqHL}: it is obtained by iteration 
of~\eqref{eq:1}, see~\cite{choque}.

We now prove Theorem \ref{theo:4}. Fix an integer $N\geq 2$.
By Abel summation, 
$$
F_{s,N}(x,t)=\sum_{k=1}^{N-1}\left(\frac{1}{k^s}-\frac{1}{(k+1)^s}\right)
\sum_{j=1}^k e^{i\pi j^2 x+2i\pi j t} + \frac1{N^s}\sum_{k=1}^N e^{i\pi k^2 x+2i\pi k t}.
$$
We set $B_m=\lceil \sqrt{Q_mQ_{m-1}}\,\rceil$ and let $r$ 
be the unique integer such that $B_{r}\le N <B_{r+1}$. We denote by 
$\vert F\vert_{s,N}(x,t)$ the sum of   
the modulus of the summands of $F_{s,N}(x,t)$. Then, for some constant $c$ independent of $N$, 
\begin{align*}
\vert F\vert_{s,N}(x,t) &\le c + \sum_{\ell=1}^r\sum_{k=B_\ell}^{B_{\ell+1}-1}
\left\vert \left(\frac{1}{k^s}-\frac{1}{(k+1)^s}\right)
\sum_{j=1}^k e^{i\pi j^2 x+2i\pi j t}\right\vert + 
\frac1{N^s}\left\vert \sum_{k=1}^{N} 
e^{i\pi k^2 x+2i\pi k t}\right\vert\\
&\ll c+ \sum_{\ell=1}^r\sum_{k=B_\ell}^{B_{\ell+1}-1}\frac{\frac k{\sqrt{Q_\ell}}+\sqrt{Q_{\ell}}}
{k^{s+1}}+\frac{\frac N{\sqrt{Q_r}}+\sqrt{Q_{r}}}{N^s}\\
&\ll c+\sum_{\ell=1}^r \frac{1}{\sqrt{Q_\ell}} \sum_{k=B_\ell}^{B_{\ell+1}-1} \frac{1}{k^{s}} 
+ \sum_{\ell=1}^r \sqrt{Q_\ell} \sum_{k=Q_\ell}^{Q_{\ell+1}-1} \frac{1}{k^{s+1}} 
+ \frac{Q_{r+1}^{\frac{1-s}2}}{Q_r^{s/2}}.
\end{align*}

For any $s>\frac 12$,
$$
\sum_{k=B_\ell}^{B_{\ell+1}-1} \frac{1}{k^{s+1}} \ll \frac{1}{B_\ell^s} 
\ll \frac{1}{(Q_{\ell} Q_{\ell-1})^{s/2}}
$$ 
but the behavior of $\sum_{k=B_\ell}^{B_{\ell+1}-1} \frac{1}{k^{s}}$ depends on whether $s=1$ 
or $s<1$.

If $s=1$, then
$$
 \sum_{k=B_\ell}^{B_{\ell+1}-1} \frac{1}{k^{s}} \ll \log(B_{\ell+1}) \ll \log(Q_{\ell+1})
$$
so that 
$$
\vert F\vert_{s,N}(x,t) \ll c+\sum_{\ell=1}^r \frac{\log(Q_{\ell+1})}{\sqrt{Q_\ell}} 
+\sum_{\ell=1}^r \frac{1}{\sqrt{Q_\ell}} +\frac{1}{\sqrt{Q_r}}
$$
and the condition 
$$
\sum_{\ell=0}^{\infty} \frac{\log(Q_{\ell+1})}{\sqrt{Q_\ell}} <\infty
$$
ensures the absolute convergence of $F_{1}(x,t)$.

If $\frac12<s<1$, then 
$$
 \sum_{k=B_\ell}^{B_{\ell+1}-1} \frac{1}{k^{s}} \ll B_{\ell+1}^{1-s} 
\ll (Q_{\ell+1}Q_{\ell})^{\frac{1-s}{2}}.
$$
Hence,
$$
\vert F\vert_{s,N}(x,t) \ll c+\sum_{\ell=1}^r \frac{Q_{\ell+1}^{\frac{1-s}2}}{Q_\ell^{s/2}} 
+\sum_{\ell=1}^r \frac{Q_{\ell}^{\frac{1-s}2}}{Q_{\ell-1}^{s/2}}  
+ \frac{Q_{r+1}^{\frac{1-s}2}}{Q_r^{s/2}}
$$
and the condition 
$$
\sum_{\ell=0}^{\infty} \frac{Q_{\ell+1}^{\frac{1-s}2}}{Q_\ell^{s/2}}<\infty
$$
ensures the absolute convergence of $F_{s}(x,t)$.

\begin{Remark} Our choice $B_m=\lceil \sqrt{Q_mQ_{m-1}}\,\rceil$ is not arbitrary. Indeed, it is 
such that
$$
\sum_{\ell=1}^{\infty}\frac{1}{\sqrt{Q_\ell}}\sum_{k=B_\ell}^{B_{\ell+1}-1} \frac{1}{k^{s}}
\quad \textup{and} \quad
\sum_{\ell = 1}^{\infty}\sqrt{Q_\ell}\sum_{k=B_\ell}^{B_{\ell+1}-1} \frac{1}{k^{s+1}}
$$ 
both converge/diverge simultaneously when $1/2<s<1$. Its importance is 
lesser when $s=1$ where we could simply 
take $B_m=Q_m$.
\end{Remark}

In the introduction, we presented Identity~\eqref{eq:TtildeT} obtained by iteration of~\eqref{eq:1} 
with the operator $T$. It is also possible to iterate~\eqref{eq:1} with Gauss' operator $G$. 
Let us assume that $x \in (0,1)$, $t\in [0,1]$, $s>\frac12$ are such that $F_s(x,t)$ is convergent. 
Then, letting $n\to +\infty$ 
in~\eqref{eq:1}, we obtain
\begin{align}
F_s(x,t) &= e^{i\frac{\pi}4} e^{-i\pi \frac{t^2}x}
x^{s-\frac12}\,F_s\Big(-\frac 1x,\frac{t}x \Big)
+ \Omega_s(x,t) \notag
\\
&=e^{i\frac{\pi}4}e^{-i\pi \frac{t^2}x}
x^{s-\frac12}\,\overline{F_s\Big(G(x), \widetilde{G}(x,t) \Big)}
+ \Omega_s(x,t), \label{eq:22}
\end{align}
with 
$$
\widetilde{G}(x,t)=\left\{\frac12\left[\frac 1x\right]-\frac tx\right\}.
$$
Skipping all the details, it can be proved that the iteration of 
(the finite version of)~\eqref{eq:22} yields the identity~\eqref{eq:GtildeG} stated in the introduction, 
which holds for all $x$ satisfying at least the 
conditions~\eqref{eq:cvQps1}-\eqref{eq:cvQps2}.

\section{Proof of Theorem~\ref{theo:1}, part $(i)$} \label{sec:newthm}

The proof is rather long and intricate. We define the following 
functions, which are building blocks of the function $\Omega_s(x,t)$:  
\begin{equation}\label{def:U}
U_s(x,t) = \int\limits_{1/2-\rho \infty}^{1/2+\rho \infty} \frac{e^{i\pi z^2 x}e^{2i\pi z \{t\}}}
{z^s(1-e^{2i\pi z})} \dd z, \\ \quad x>0, t\in \mathbb  R,\, s\ge 0
\end{equation}
\begin{equation}\label{def:V}
V_s(x,t) = \rho x^{s-1/2} \sum_{k=1}^{\infty}e^{-i\pi (k-\{t\})^2/x}
\Big(\frac1{(k-\{t\})^s}-\frac1{k^s}\Big), \quad x>0, \,t\in \mathbb  R, s\ge 0
\end{equation}
\begin{multline*}\widehat{W}_s(x,t,u) =\sum_{k=1}^{\infty} e^{-i\pi (k-\{t\})^2/x} 
\bigg(\frac{1}{(\rho xu+k-\{t\})^s}-\frac{1}{(k-\{t\})^s}\bigg), \\ 
\quad x>0, \,t\in \mathbb  R, s\ge 0,\, 
u\in \mathbb  R
\end{multline*}
and 
\begin{equation}\label{def:W}
W_s(x,t) =\rho x^s\int\limits_{-\infty}^{\infty} e^{-\pi x u^2}
\widehat{W}_s(x,t,u), \quad x>0,\, t\in \mathbb  R,\, s\ge 0. 
\end{equation}

Due to the identity $\int\limits_{-\infty}^{\infty} e^{-\pi x^2 u}\dd u=1/\sqrt{x}$ 
($x>0$), it is easy to see that
\begin{equation}\label{eq:riv13}
V_s(x,t)+W_s(x,t)=\rho x^s\int\limits_{-\infty}^{\infty} e^{-\pi x u^2} 
\sum_{k=1}^{\infty} e^{-i\pi (k-\{t\})^2/x} 
\bigg(\frac{1}{(\rho xu+k-\{t\})^s}-\frac{1}{k^s}\bigg)\dd u.
\end{equation}

\subsection{Structure of the proof}\label{ssec:structure}

\

In this section, we present all the details of the proof of the theorem 
except the proofs of four lemmas, which are postponed to 
Section~\ref{ssec:lemmas}. Throughout, we assume that $x>0$ and explain in the 
end how to get the case $x<0$. The method is borrowed to 
Mordell~\cite{mordell} in the version presented in~\cite{choque}.

We introduce the parameters $\xi=t-\lfloor (n-\frac12)x+ t\rfloor$ and 
$\lambda=\lfloor (n-\frac12)x+t\rfloor- \lfloor t\rfloor$. Note that 
$\lambda\ge 0$ and $\xi+\lambda = \{ t\} \in [0,1)$. 
 
Let us define the function
$$
g_s(z)=\frac{1}{z^s}e^{i \pi z^2 x+2i\pi z \xi}
$$
which is holomorphic as a function of $z$ in $\mathbb C\setminus (-\infty,0]$. 

\begin{lem}\label{lem:1} For all $n\ge 1$, $s\ge 0$, $x>0$ and $t\in \mathbb R$, we have
$$
F_{s,n-1}(x,t) 
=   \int\limits_{-1/2-\rho \infty}^{-1/2+\rho \infty} 
\frac{g_s(z+1)}{1-e^{2i\pi z}} \dd z -\int\limits_{-1/2-\rho \infty}^{-1/2+\rho \infty} 
\frac{g_s(z+n)}{1-e^{2i\pi z}} \dd z.
$$
\end{lem}
We focus on the first integral in Lemma~\ref{lem:1}, i.e.
$$
\int\limits_{-1/2-\rho \infty}^{-1/2+\rho \infty} \frac{g_s(z+1)}{1-e^{2i\pi z}} \dd z
$$
We have
\begin{equation}\label{eq:5}
\frac{g_s(z+1)}{1-e^{2i\pi z}} =\sum_{k=0}^{\lambda-1} g_s(z+1)e^{2i\pi kz}
+\frac{e^{2i\pi\lambda z}}{1-e^{2i \pi z}}g_s(z+1),
\end{equation}
where the assigned value of $\lambda$ is in fact irrelevant. If $\lambda=0$, the 
sum is empty, equal to $0$ and~\eqref{eq:5} is a tautology;~to avoid talking 
about empty sums, we assume from now on that $\lambda\ge 1$ but the results also hold 
when $\lambda=0$. Now
\begin{align*}
J_k:&=\int\limits_{-1/2-\rho \infty}^{-1/2+\rho \infty} g_s(z+1)e^{2i\pi k z} \dd z 
= \int\limits_{1/2-\rho \infty}^{1/2+\rho \infty} g_s(z)e^{2i\pi k z} \dd z
=\int\limits_{1/2-\rho \infty}^{1/2+\rho \infty} \frac{1}{z^s}e^{i\pi z^2x+2i\pi (k+\xi) z} \dd z
\\ &= \rho e^{-i\pi(\xi+k)^2/x}\int\limits_{-\infty}^{\infty} 
  \frac{e^{i\pi x(\rho u+\frac12+\frac{\xi+k}x)^2}}{(\rho u+\frac12)^s} \dd u
 =\rho^{1-s}e^{-i\pi(\xi+k)^2/x}\int\limits_{-\infty}^{\infty} 
 \frac{e^{-\pi x(u+\frac\rhob2+\rhob\frac{\xi+k}x)^2}}{(u+\frac{\rhob}2)^s} \dd u.
\end{align*}

We set $v=\frac\rhob2+\rhob\frac{\xi+k}x$. We have that $\xi+k<0$ for each
integer $k \in\{0, \ldots \lambda-1\}$, so that
\begin{align}
J_k = \rho^{1-s}e^{-i\pi(\xi+k)^2/x}\int\limits_{v-\infty}^{v+\infty} 
\frac{e^{-\pi x u^2}}{(u-\rhob \frac{\xi+k}{x})^s} \dd u \notag
\\ =\rho^{1-s} e^{-i\pi(\xi+k)^2/x}\int\limits_{-\infty}^{\infty} 
\frac{e^{-\pi x u^2}}{(u-\rhob \frac{\xi+k}{x})^s} \dd u \label{eq:4}
\end{align}
where the second equality holds by Cauchy theorem because $\rhob \frac{\xi+k}{x}$ is never 
in the closed horizontal strip defined by the lines 
$\textup{Im}(z)=0$ and $\textup{Im}(z)=\textup{Im}(v)$. (Indeed, 
we have simultaneously 
$\textup{Im}(\rhob \frac{\xi+k}{x})>\textup{Im}(v)$ and $\textup{Im}(\rhob \frac{\xi+k}{x})>0$.)

We now rewrite~\eqref{eq:4} as 
\begin{align*}
J_k &=\rho e^{-i\pi(\xi+k)^2/x}\left(-\frac{x}{\xi+k}\right)^s\int\limits_{-\infty}^{\infty} e^{-\pi x u^2} \dd u
\\
& \qquad +\rho^{1-s} e^{-i\pi(\xi+k)^2/x}
\int\limits_{-\infty}^{\infty}e^{-\pi x u^2} 
\left(\frac{1}{(u-\rhob\frac{\xi+k}{x})^s} -\frac{1}{(-\rhob\frac{\xi+k}{x})^s}
\right)
\dd u
\\
&= \rho e^{-i\pi(\xi+k)^2/x}\frac{x^{s-1/2}}{(-\xi-k)^s}
+ \widetilde{J}_k,
\end{align*}
where $\widetilde{J}_k$ is the integral on the second line. We observe here that 
$$
\vert \widetilde{J}_k \vert \ll 
 \frac{1}{\left\vert \textup{Im}(\rhob\frac{\xi+k}{x})\right\vert^{s+1}} \cdot
\int\limits_{-\infty}^{\infty} \vert u\vert e^{-\pi x u^2} \dd u \ll \frac{x^s}{\vert \xi+k\vert^{s+1}}
$$
where the implicit constant is absolute. 

Integrating~\eqref{eq:5}, we thus obtain
\begin{eqnarray}
\nonumber&& \int\limits_{-1/2-\rho \infty}^{-1/2+\rho \infty} \frac{g_s(z+1)}{1-e^{2i\pi z}} \dd z \\
\label{eq:6} && =\rho x^{s-1/2}\sum_{k=0}^{\lambda-1}\frac{e^{-i\pi(\xi+k)^2/x}}{(-\xi-k)^s}
+\sum_{k=0}^{\lambda-1} \widetilde{J}_k +\int\limits_{-1/2-\rho \infty}^{-1/2+\rho \infty} 
\frac{g_s(z+1)e^{2i\pi \lambda z}}{1-e^{2i\pi z}} \dd z.
\end{eqnarray}
We will treat the three expressions in~\eqref{eq:6} separately.

The third term is simple because our choice of 
$\lambda =\lfloor (n-\frac12)x +t\rfloor-\lfloor t\rfloor$ ensures that
$$\int\limits_{-1/2-\rho \infty}^{-1/2+\rho \infty} 
\frac{g_s(z+1)e^{2i\pi \lambda z}}{1-e^{2i\pi z}} \dd z
=\int\limits_{1/2-\rho \infty}^{1/2+\rho \infty} 
\frac{e^{i\pi z^2 x}e^{2i\pi z \{t\}}}{z^s(1-e^{2i\pi z})} \dd z
$$
is independent of $n$ and is equal to the function $U_s(x,t)$ defined in~\eqref{def:U}.

To study the first term in~\eqref{eq:6}, we change the summation index $k$ to $\lambda-k$: 
\begin{align*}
\rho x^{s-1/2} \sum_{k=0}^{\lambda-1}\frac{e^{-i\pi(\xi+k)^2/x}}{(-\xi-k)^s}
&=\rho x^{s-1/2}\sum_{k=1}^{\lambda} \frac{e^{-i\pi (  \{t\}-k)^2/x}}{(k-\{t\})^s}
\\
&=\rho x^{s-1/2} e^{-i\pi \{t\}^2/x}F_{s,\lambda}\Big(-\frac1x,\frac{\{t\}}{x}\Big)
\\
&\qquad  + \rho x^{s-1/2} \sum_{k=1}^{\lambda}e^{-i\pi (k-\{t\})^2/x}
\Big(\frac1{(k-\{t\})^s}-\frac1{k^s}\Big).
\end{align*}
The sum is a partial sum of the series $V_s(x,t)$ defined in~\eqref{def:V}. If $s=0$, 
then the summand is equal to $0$. If $s>0$, we need the following lemma.
\begin{lem}\label{lem:riv1}
For all $N\ge 0$, $s>0$, $x>0$ and $t\in \mathbb R$, we have
$$
V_s(x,t)=\rho x^{s-1/2} \sum_{k=1}^{N}e^{-i\pi (k-\{t\})^2/x}
\left(\frac1{(k-\{t\})^s}-\frac1{k^s}\right)+ 
\mathcal{O}\left(\frac{x^{s-\frac12}}{(N+1-\{t\})^s}\right)
$$
where the implicit constant is absolute. 
\end{lem}
Therefore, 
\begin{multline}\label{eq:riv4}
\rho x^{s-1/2} \sum_{k=0}^{\lambda-1}\frac{e^{-i\pi(\xi+k)^2/x}}{(-\xi-k)^s} 
\\
=\rho x^{s-1/2} e^{-i\pi \{t\}^2/x}F_{s,\lambda}\Big(-\frac1x,\frac{\{t\}}{x}\Big)
+V_s(x,t)+\mathcal{O}\left(\frac{x^{s-\frac12}}{(\lambda+1-\{t\})^s}\right)
\end{multline}
where the implicit constant depends on $s$.

To study the second term in~\eqref{eq:6}, we do similar formal 
manipulations and obtain the representation
$$
\sum_{k=0}^{\lambda-1} \widetilde{J}_k
= \rho x^s\int\limits_{-\infty}^{\infty} e^{-\pi x u^2} 
\left(\sum_{k=1}^{\lambda} e^{-i\pi (k-\{t\})^2/x} \Big(\frac{1}{(\rho xu+k-\{t\})^s}
-\frac{1}{(k-\{t\})^s}\Big)\right)\dd u ,
$$
where the sum in the integrand is a partial sum of $\widehat{W}_s(x,t,u)$. If $s=0$, 
then the summand is equal to $0$. If $s>0$, we need the following lemma.
\begin{lem} \label{lem:riv2}
For all $N\ge 0$, $s>0$, $x>0$ and $t\in \mathbb R$, we have
\begin{equation*}
\widehat{W}_s(x,t,u)=\sum_{k=1}^{N} e^{-i\pi (k-\{t\})^2/x} 
\bigg(\frac{1}{(\rho xu+k-\{t\})^s}-\frac{1}{(k-\{t\})^s}\bigg) 
+\mathcal{O}\left(\frac{\vert u x\vert}{(N+1-\{t\})^s}
\right),
\end{equation*}
where the implicit constant is effective and depends at most on $s$.
\end{lem}
It follows that
\begin{align}
\sum_{k=0}^{\lambda-1} \widetilde{J}_k 
&=W_s(x,t)
+ 
\mathcal{O}\left(\frac{x^{s+1}}{(\lambda+1-\{t\})^s}
\int\limits_{-\infty}^{\infty}\vert u\vert e^{-\pi xu^2}\dd u \right) \notag
\\
&=W_s(x,t)
+ 
\mathcal{O}\left(\frac{x^{s}}{(\lambda+1-\{t\})^s}\right) \label{eq:riv6}
\end{align}
because 
$$
\int\limits_{-\infty}^{\infty}\vert u\vert e^{-\pi xu^2}\dd u = \frac{1}{\pi x}.
$$

We now use the estimates~\eqref{eq:5}, \eqref{eq:riv4}, \eqref{eq:riv6} 
in~\eqref{eq:6} together with Lemma~\ref{lem:3}. This gives us 
\begin{multline*}
F_{s,n-1}(x,t)=\rho x^{s-\frac12} e^{-i\pi \{t\}^2/x} F_{s,\lambda}(-1/x,t)
+ U_s(x,t)+V_s(x,t)+ W_s(x,t) 
\\+ 
 \int\limits_{-1/2-\rho \infty}^{-1/2+\rho \infty} \frac{g_s(z+n)}{e^{2i\pi z}-1} \dd z 
+\mathcal{O}\left(\frac{x^{s-\frac12}}{(\lambda+1-\{t\})^s}\right)
+\mathcal{O}\left(\frac{x^{s}}{(\lambda+1-\{t\})^s}\right).
\end{multline*}
Since $0<x<1$, the first error term absorbs the second one. 
We also want to replace $F_{s,\lambda}(-1/x,\{t\}/x)$ 
by $F_{s,\lfloor (n-1)x\rfloor}(-1/x,\{t\}/x)$. This can be done 
at the cost of an error 
$$
x^{s-\frac12}\sum_{k=1+\lfloor (n-1)x\rfloor}^{\lambda} \frac{1}{k^s} 
\le  \frac{2x^{s-\frac12}}{(1+\lfloor (n-1)x\rfloor)^s} 
$$
because $\lfloor (n-1)x\rfloor\le \lambda\le \lfloor (n-1)x\rfloor+2$. Hence
\begin{multline*}
F_{s,n-1}(x,t)=\rho x^{s-\frac12} e^{-i\pi \{t\}^2/x} F_{s,\lfloor (n-1)x\rfloor}(-1/x,\{t\}/x)
+ U_s(x,t)+V_s(x,t)+ W_s(x,t) 
\\+ 
 \int\limits_{-1/2-\rho \infty}^{-1/2+\rho \infty} \frac{g_s(z+n)}{e^{2i\pi z}-1} \dd z 
+\mathcal{O}\left(\frac{x^{s-\frac+2}}{(\lambda+1-\{t\})^s}\right)+\mathcal{O} 
\left(\frac{x^{s-\frac12}}{(1+\lfloor (n-1)x\rfloor)^s}\right).
\end{multline*}
It remains  to deal with the second integral in Lemma~\ref{lem:1}. For this, 
we have to distinguish between the case $s=1$ and the case $s\neq 1$.
\begin{lem}\label{lem:3} 
If $s\ge 0, s\neq 1$, for any $n\ge1$, $x>0$, $t\in\mathbb R$, we have
$$
\left\vert\int\limits_{-1/2-\rho \infty}^{-1/2+\rho \infty}
 \frac{g(z+n)}{1-e^{2i\pi z}} \dd z \right\vert
\ll \min\Big(\frac{1}{n^s\sqrt{x}}, \frac{1}{x^{\frac{1-s}{2}}}\Big)
$$
where the implicit constant depends on $s$. 

If $s=1,$ for any $n\ge 1$, $x>0$, $t\in\mathbb R$, we have
$$
\left\vert\int\limits_{-1/2-\rho \infty}^{-1/2+\rho \infty}
 \frac{g(z+n)}{1-e^{2i\pi z}} \dd z \right\vert
\ll \min\Big(\frac{1}{n\sqrt{ x}}, 1+\vert \log(n \sqrt{x})\vert\Big)
$$
where the implicit constant is absolute.
\end{lem}
In the case $s\neq 1$, the threshold $n\sqrt{x}=1$ determine which bound is the best.

We now replace $n$ by $n+1$ and set $I_s(x,t)=U_s(x,t)+V_s(x,t)+ W_s(x,t)$. We get 
\begin{multline}\label{eq:riv7}
F_{s,n}(x,t)=\rho x^{s-\frac12} e^{-i\pi \{t\}^2/x} F_{s,\lfloor nx\rfloor}(-1/x,\{t\}/x)
+ I_s(x,t)
\\
+
\mathcal{O}\left(\frac{x^{s-\frac12}}{(\lfloor (n+\frac12)x+t\rfloor+1-t)^s}\right)
+\mathcal{O} \left(\frac{x^{s-\frac12}}{(1+\lfloor n x\rfloor)^s}\right) 
\\
+ 
\begin{cases}
\mathcal{O} \left( \min\Big(\frac{1}{(n+1)^s\sqrt{x}}, \frac{1}{x^{\frac{1-s}{2}}}\Big) \right) 
\quad \textup{if} \; s\ge 0, s\neq 1
\\
\mathcal{O} \left( \min\Big(\frac{1}{(n+1)\sqrt{x}}, 1+\vert \log((n+1) \sqrt{x})\vert\Big)\right) 
\quad \textup{if} \; s=1.
\end{cases}
\end{multline}
For any $x>0$, any $n\ge 0$ and any $t\in \mathbb  R$, 
$\lfloor (n+\frac12)x+t\rfloor+1-t$ and $ \lfloor nx\rfloor+1$ are both 
$\ge \lfloor nx+t\rfloor+1-t$
so that we can simplify the error term in~\eqref{eq:riv7} to just
$$
\mathcal{O} \left(\frac{x^{s-\frac12}}{(\lfloor nx+t\rfloor+1-t)^s}\right) + 
\begin{cases}
\mathcal{O} \left( \min\Big(\frac{1}{(n+1)^s\sqrt{ x}}, 
\frac{1}{x^{\frac{1-s}{2}}}\Big) \right) \quad \textup{if} \; s\ge 0, s\neq 1
\\
\mathcal{O} \left( \min\Big(\frac{1}{(n+1)\sqrt{ x}}, 
1+\vert \log((n+1) \sqrt{x})\vert\Big)\right) \quad \textup{if} \; s=1.
\end{cases}
$$

To deal with the case $x<0$, we simply take the 
complex conjugate of both sides of~\eqref{eq:riv7} and change $x$ to $-x$ and $t$ to $-t$. We finally 
deduce from all this discussion that {\em for any $x\in (-1,1)$, $x\neq 0$, any $t\in \mathbb R$, 
any $s\ge 0$ and any $n\ge 0$}, we have
\begin{multline*} F_{s,n}(x,t)=\rho \vert x \vert^{s-\frac12} e^{-i\pi \{\sigma(x)t\}^2/x} 
F_{s,\lfloor n \vert x\vert\rfloor}\Big(-\frac1x,\frac{\{\sigma(x)t\}}x\Big)
+ \Omega_s(x,t)
\\+
\mathcal{O} \left(\frac{\vert x\vert ^{s-\frac12}}
{(\lfloor n\vert x\vert+\sigma(x)t\rfloor+1-\sigma(x)t)^s}\right)
\\ 
+ 
\begin{cases}
\mathcal{O} \left( \min\Big(\frac{1}{(n+1)^s\sqrt{\vert x\vert}}, 
\frac{1}{\vert x\vert^{\frac{1-s}{2}}}\Big) \right) \quad \textup{if} \; s\ge 0, s\neq 1
\\
\mathcal{O} \left( \min\Big(\frac{1}{(n+1)\sqrt{ \vert x\vert}}, 
1+\vert \log((n+1)\sqrt{\vert x\vert }  )\vert\Big)\right) \quad \textup{if} \; s=1.
\end{cases}
\end{multline*}
where $\Omega_s(x,t)=I_s(x,t)$ is $x>0$ and $\Omega_s(x,t)=\overline{I_s(-x,-t)}$ if $x<0$. 
The implicit constants are effective and depend at most on $s$. We obtain the expression for 
$I_s(x,t)$ given in the introduction by means of the expression~\eqref{eq:riv13} for $V_s(x,t)+W_s(x,t)$.

\subsection{Proofs of the lemmas}\label{ssec:lemmas}

\subsubsection{Proof of Lemma~\ref{lem:1}}

\

We follow Mordell's method as it is presented in the book~\cite{choque}.
We define 
$$
f_s(z)=\frac{1}{e^{2i\pi z}-1}\sum_{k=1}^{n-1} g_s(z+k).
$$
and we integrate it over the parallelogram $ADCB$ (positively oriented) defined by
$A=\frac 12+\rho d$, $B=\frac 12 -\rho d$, $C=-\frac12-\rho d$, $D=-\frac12+\rho d$ and $d>0$ 
is a parameter. Clearly, 
$$
\int_{ABCD} f(z) \dd z = 2i\pi \textup{Res}(f_s(z), z=0) = \sum_{k=1}^{n-1} g_s(k) = F_{s,n-1}(x,t).
$$
If $z$ is on $CB=\{-\frac12+\rho d u, -d\le u\le d\}$ or 
$AD=\{u+\rho d , -\frac12\le u\le \frac12\}$, we have 
\begin{equation}\label{eq:riv100}
\textup{Re}(i\pi (z+k)^2+2i\pi(z+k)t)=-\pi d^2 x\pm \sqrt{2} \pi \cdot d((u+k)x+t) 
= -\pi d^2 x+\mathcal{O}(d)
\end{equation}
where the implicit constant does not depend on $u$. Moreover, $\vert e^{2i \pi z}\vert =
e^{\sigma \sqrt{2}\pi d}$ with $\sigma=1$ on $CB$ and $\sigma=-1$ on $AD$, so that 
\begin{equation}\label{eq:riv101}
\lim_{d\to +\infty} \vert e^{2i\pi z}-1\vert 
=
\begin{cases} 
+\infty, \quad z\in CB
\\
1, \ \ \  \, \quad z\in AD.
\end{cases}
\end{equation}
It follows from~\eqref{eq:riv100} and~\eqref{eq:riv101} that 
$$
\lim_{d\to +\infty} \int_{CB\cup AD} f_s(z) \dd z =0.
$$
Thus
$$
F_{s,n-1}(x,t) = \int_{BA} f_s(z) \dd z - \int_{CD} f_s(z) \dd z + o(1),
$$
where $o(1)$ is for $d\to +\infty$. We observe that $BA=CD+1$, so that
\begin{align*}
F_{s,n-1}(x,t) &= \int_{CD} f_s(z+1) \dd z - \int_{CD} f_s(z) \dd z + o(1)
\\
&= \int_{CD} (f_s(z+1)-f_s(z))\, \dd z + o(1)
\\
&= \int_{CD} \frac{g_s(z+n)-g_s(z+1)}{e^{2i\pi z}-1} \,\dd z + o(1).
\end{align*}
We now let $d\to +\infty$ to get the lemma.

\subsubsection{Proof of Lemma~\ref{lem:riv1}}
\

By definition of $V_s(x,t)$, the point is to estimate the series
$$
\rho x^{s-1/2} \sum_{k=N+1}^{\infty}e^{-i\pi (k-\{t\})^2/x}
\left(\frac1{(k-\{t\})^s}-\frac1{k^s}\right),
$$
the modulus of which is obviously bounded by
$$
x^{s-1/2}\sum_{k=N+1}^{\infty} \left\vert \frac1{(k-\{t\})^s}-\frac1{k^s}\right\vert.
$$
We observe that, since $\{t\} \in [0,1)$ and $N\ge 0$, 
\begin{align*}
\sum_{k=N+1}^{\infty} \left\vert \frac1{(k-\{t\})^s}-\frac1{k^s}\right\vert &= 
\sum_{k=N+1}^{\infty} \left(\frac1{(k-\{t\})^s}-\frac1{k^s}\right)
\\
&\le 
\sum_{k=N+1}^{\infty} \left( \frac1{(k-\{t\})^s}-\frac1{(k+1-\{t\})^s}\right)
\\
&=\frac{1}{(N+1-\{t\})^s}.
\end{align*}
The lemma follows.

\subsubsection{Proof of Lemma~\ref{lem:riv2}}

\

We want to estimate the modulus of 
$$
\sum_{k=N+1}^{\infty} e^{-i\pi (k-\{t\})^2/x} 
\bigg(\frac{1}{(\rho xu+k-\{t\})^s}-\frac{1}{(k-\{t\})^s}\bigg),
$$ 
which is bounded above by
$$
\sum_{k=1}^{\infty}  
\left \vert \frac{1}{(\rho xu+k+N-\{t\})^s}-\frac{1}{(k+N-\{t\})^s}\right\vert
$$
This problem is similar to the one we  dealt with in the proof of 
Lemma~\ref{lem:riv1} but this is a bit more difficult 
here because $\rho xu$ is not real. We set $M=N-\{t\}$.
We observe that
\begin{align*}
\left \vert \frac{1}{(\rho xu+k+M)^s}-\frac{1}{(k+M)^s}\right\vert &=
\left \vert \frac{1}{(\rho xu+k+M)^s}\cdot\left(1-\Big(1+\frac{\rho x u}{k+M}\Big)^s\right)\right\vert
\\
& \ll 
\begin{cases} \displaystyle 
\frac{\vert x u\vert}{\vert k+M+\rho x u\vert^s(k+M)} \quad \textup{if} \quad \vert x u\vert\le k+M
\\
 \displaystyle 
\frac{\vert x u\vert^s}{\vert k+M+\rho x u\vert^s(k+M)^s} \quad \textup{if} \quad \vert x u\vert\ge k+M
\end{cases}
\end{align*}
where the implicit constant depends on $s$ only. Hence 
\begin{multline*}
\sum_{k=1}^{\infty}
\left \vert \frac{1}{(\rho xu+k+M)^s}-\frac{1}{(k+M)^s}\right\vert 
\\
\ll \sum_{\stackrel{k=1}{\vert x u\vert\le k+M}}^{\infty} 
\frac{\vert x u\vert}{\vert k+M+\rho x u\vert^s(k+M)} + \sum_{\stackrel{k=1}{\vert x u\vert\ge k+M}}^{\infty} 
\frac{\vert x u\vert^s}{\vert k+M+\rho x u\vert^s(k+M)^s}.
\end{multline*}
We denote by $S_1$ and $S_2$ these two series and 
we write $\rho xu = \frac{\sqrt{2}}{2}(v+iv)$ for some $v\in \mathbb R$. 
(Note that $\vert x u\vert =\vert v\vert$.)

If $v\ge 0$, then 
$$
0\le S_1\le \sum_{k=1}^{\infty} \frac{v}{(k+M+v/\sqrt{2})^s(k+M)} \le 
\sum_{k=1}^{\infty} \frac{v}{(k+M)^{s+1}}\ll \frac{v}{(M+1)^s}
$$
for some implicit constant that depends only on $s$. Moreover, 
\begin{equation*}
0\le S_2\le \sum_{1\le k\le v} \frac{v^s}{(k+M+v/\sqrt{2})^s(k+M)^{s}}
\ll \sum_{1\le k \le v} \frac{v^s }{(k+M)^{2s}} 
\le  \frac{v^{s+1} }{(M+1)^{2s}}
\end{equation*}
for some implicit constant that depends only on $s$. Adding 
the two upper bounds for $S_1$ and $S_2$, we get
$$
\sum_{k=1}^{\infty}
\left \vert \frac{1}{(\rho xu+k+M)^s}-\frac{1}{(k+M)^s}\right\vert \ll \frac{\vert xu\vert}{(M+1)^s} 
$$
which proves the lemma in this case.

Let us now consider the case $v\le 0$. If $\vert v\vert\le k+M$, then  
$$
\vert k+M+\rho x u \vert \ge (k+M-\vert  v\vert /\sqrt{2})\ge (1-1/\sqrt{2})(k+M)
$$
so that
$$
0\le S_1\ll \sum_{k=1}^{\infty} \frac{\vert v\vert }{(k+M)^{s+1}}\ll \frac{\vert x u\vert}{(M+1)^s}
$$
for some implicit constants that depend only on $s$. If $\vert v\vert \ge k+M$, then 
$\vert k+M+\rho x u\vert^2\ge w^2/2\ge (k+M)^2/2$ so that again
$$
0\le S_2\ll  \sum_{1\le k \le \vert v\vert} \frac{\vert v\vert^s }{(k+M)^{2s}}
\ll \frac{\vert x u\vert^{s+1}}{(M+1)^{2s}}
$$
for some implicit constants that depend on $s$. We conclude exactly as above.

\subsubsection{Proof of Lemma~\ref{lem:3}}

\

Again, we follow Mordell's method in the book~\cite{choque}.
For any $s\ge 0$ and $z=-\frac12+\rho u 
\in -\frac12 +\rho\mathbb R$, we have
$$
\left\vert \frac{g_s(z+n)}{e^{2i\pi z}-1}  \right\vert = \frac{e^{-\pi x u^2}}
{\vert u+ \rhob (n-\frac12) \vert^s}\cdot 
\frac{e^{-\sqrt{2}\pi u \theta}}{\vert e^{2i\pi z}-1\vert}
$$
with $\theta=(n-\frac12)x+t+ \xi$. By definition of $\xi$, 
we have $0\le \theta \le 1$ and it follows that
$$
\frac{e^{-\sqrt{2}\pi u \theta}}{\vert e^{2i\pi z}-1\vert} =\mathcal{O}(1)
$$
for any $u \in \mathbb R$. Moreover, for any $u\in \mathbb R$ and any $n\ge 1$, we have
$$
\frac{1}{\vert u+ \rhob (n-\frac12) \vert^s} \ll \frac{1}{\vert u+ \rhob n \vert^s}
$$
for some effective constant that depends only on $s$.
Hence, for any $s\ge 0$, 
$$
\left\vert\int\limits_{-1/2-\rho \infty}^{-1/2+\rho \infty}
 \frac{g_s(z+n)}{e^{2i\pi z}-1} \dd z \right\vert \ll
\int\limits_{-\infty}^{+\infty} \frac{e^{-\pi x u^2}}{\vert u+ \rhob n \vert^s } \dd u
$$
where the implicit constant depends on $s$. 
Since $\rhob \not \in \mathbb R$, we readily deduce that
$$
\left\vert\int\limits_{-1/2-\rho \infty}^{-1/2+\rho \infty}
 \frac{g_s(z+n)}{e^{2i\pi z}-1} \dd z \right\vert \le  \frac{c_1(s)}{n^s}
\int\limits_{-\infty}^{+\infty} e^{-\pi x u^2} \dd u= \frac{c_1(s)}{n^s\sqrt{x}}.
$$
for some effective constant $c_1(s)$. 

To get the second upper bound, we need to distinguish the case $s=1$ and the case $s\neq 1$. 

\medskip

\noindent {\bf Case $s\ge 0, s\neq 1$.} We  assume for the moment 
that $0\le n\sqrt{x}\le 1$ and explain 
below how to remove this assumption. We set $y=n\sqrt{x/2}$; with $v=\sqrt{x}u$, 
we get
\begin{align*}
\int\limits_{-\infty}^{+\infty} \frac{e^{-\pi x u^2}}{\vert u+ \rhob n \vert^s } \dd u 
&\le \sqrt{2} 
x^{\frac{s-1}{2}}\int\limits_{-\infty}^{+\infty} \frac{e^{-\pi v^2}}{(\vert v +y\vert + y)^s } \dd v 
\\
&\le \sqrt{2} x^{\frac{s-1}{2}}\int\limits_{0}^{+\infty} \frac{e^{-\pi v^2}}{(v + y)^s} \dd v 
+\sqrt{2} x^{\frac{s-1}{2}}\int\limits_{0}^{+\infty} \frac{e^{-\pi v^2}}{(\vert v - y\vert +y)^s} \dd v. 
\end{align*}
First, 
\begin{align*}
\int\limits_{0}^{+\infty} \frac{e^{-\pi v^2}}{(v + y)^s} \dd v 
&\le 
\int\limits_{0}^{1} \frac{\dd v}{(v + y)^s}  + \frac{1}{(1+y)^s}
\int\limits_{1}^{+\infty} e^{-\pi v^2} \dd v
\\
& \le \frac{(1+y)^{1-s}-y^{1-s}}{1-s}  + \frac{1}{(1+y)^s} \le c_2(s)
\end{align*}
for some effective constant $c_2(s)$ because $0\le y\le \sqrt{2}/2$.
Second, 
\begin{align*}
\int\limits_{0}^{+\infty} \frac{e^{-\pi v^2}}{(\vert v - y\vert+y) ^s} \dd v 
&\le 
\int\limits_{0}^{y} \frac{\dd v}{(2y-v)^s} + \int_y^{y+1}\frac{\dd v}{v^s}+ 
\frac{1}{(1+y)^s}\int\limits_{1}^{+\infty} e^{-\pi v^2} \dd v
\\
& \le \frac{2^{1-s}-1}{1-s} y^{1-s} + \frac{(1+y)^{1-s}-y^{1-s}}{1-s} 
+ \frac{1}{(1+y)^s} \le c_3(s)
\end{align*}
for some effective constant $c_3(s)$ because $0\le y\le \sqrt{2}/2$. Hence, 
$$
\int\limits_{-\infty}^{+\infty} \frac{e^{-\pi x u^2}}{\vert u+ \rhob n \vert^s } \dd u \le 
c_4(s)x^{\frac{s-1}{2}}.
$$
with $c_4(s)=\sqrt{2} (c_2(s)+c_3(s))$.

In summary, we have obtained so far:  
$$
\left\vert\int\limits_{-1/2-\rho \infty}^{-1/2+\rho \infty}
 \frac{g_s(z+n)}{e^{2i\pi z}-1} \dd z \right\vert \le 
\begin{cases} \frac{c_1(s)}{n^s\sqrt{x}} \;\textup{for all} \;n\ge 1,x >0
\\
c_4(s)x^{\frac{s-1}{2}} \; \textup{if} \;0\le n\sqrt{x}\le 1.
\end{cases}
$$
With $c(s)=\max(c_1(s), c_4(s))$ and since 
$1/(n^s\sqrt{x})\ge x^{\frac{s-1}{2}}$ when $0\le n\sqrt{x}\le 1$, we deduce that
$$
\left\vert\int\limits_{-1/2-\rho \infty}^{-1/2+\rho \infty}
 \frac{g_s(z+n)}{e^{2i\pi z}-1} \dd z \right\vert 
\le c(s)\min\Big(\frac{1}{n^s\sqrt{ x}}, \frac{1}{x^{\frac{1-s}{2}}}\Big)
$$
for all $n\ge 1,x >0$. This completes the proof in this case.

\medskip

\noindent {\bf Case $s=1$.} 
We set $y=n \sqrt{x/2}\ge 0$ for simplicity. We then have (with $v=\sqrt{x}u$)
$$
\int\limits_{-\infty}^{+\infty} \frac{e^{-\pi x u^2}}{\vert u+ \rhob n \vert } \dd u 
\ll \int\limits_{0}^{+\infty} \frac{e^{-\pi v^2}}{v+ y} \dd v
+ \int\limits_{-\infty}^{0} \frac{e^{-\pi v^2}}{\vert v+ y\vert+y} 
\dd v.
$$
First, 
\begin{align*}
0\le \int\limits_{0}^{+\infty} \frac{e^{-\pi v^2}}{v+ y} \dd v &\le  
\int\limits_{0}^{1} \frac{\dd v}{v + y}  + \frac{1}{1+y}\int \limits_{1}^{+\infty} e^{-\pi v^2} \dd v
\\&\le  \log(y+1) -\log(y)+\int \limits_{0}^{\infty} e^{-\pi v^2} \dd v
\\&= \log(1+n\sqrt{x/2})-\log(n\sqrt{x/2}) + 1.
\end{align*}
Second, 
\begin{align*}
0\le \int\limits_{-\infty}^{0} \frac{e^{-\pi v^2}}{y+\vert v+ y\vert} \dd v &= 
\int\limits_0^{+\infty} \frac{e^{-\pi v^2}}{y+\vert v- y\vert} \dd v
 \\
&\le \int\limits_{0}^{y} \frac{\dd v}{2y-v}  + \int \limits_{y}^{y+1} \frac{\dd v}{v} 
+ \frac1{y+1}\int \limits_{y+1}^{\infty} e^{-v} \dd v
\\
&\le \log(2)+\log(y+1)-\log(y)+1 
\\
&=\log(1+n\sqrt{x/2}) -\log(n\sqrt{x/2}) + 1+\log(2).
\end{align*}
Collecting both estimates, we obtain
$$
\left\vert\int\limits_{-1/2-\rho \infty}^{-1/2+\rho \infty}
 \frac{g(z+n)}{e^{2i\pi z}-1} \dd z \right\vert \ll \vert \log (n\sqrt{x})\vert+\log(1+n\sqrt{x/2})+1
$$
for some absolute constant. 
If $n\sqrt{x}\ge 1$, then 
$\log(1+n\sqrt{x/2})\le \log(n\sqrt{x}) + 1$ 
and if $0<n\sqrt{x}\le 1$, then $\log(1+n\sqrt{x/2}) \le \log(1+1/\sqrt{2})$. 
Consequently, there exists an absolute constant $c$ such that
$$
\left\vert\int\limits_{-1/2-\rho \infty}^{-1/2+\rho \infty}
 \frac{g(z+n)}{e^{2i\pi z}-1} \dd z \right\vert \le c(\vert \log (n\sqrt{x})\vert+1)
$$
 for any $n\ge 1$ and any $x>0$.

\section{Proof of Theorem~\ref{theo:1}, part $(ii)$}\label{sec:2}

In this section, we prove that  
\begin{itemize}
\item the function $\Omega_s(x)$ is continuous on $\mathbb {R}\setminus\{0\}$ for any $s\ge 0$.

\item the function $\Omega_s(x)$ is differentiable at any rational number $p/q$ with $p,q$ both odd, 
for any $s\ge 0$.

\item  the function $\Omega_s(x)-\frac{\rho^{1-s}\Gamma(\frac{1-s}{2})}{2\pi^{\frac{1-s}{2}}}
\vert x\vert^{\frac{s-1}2}$ is bounded on $\mathbb  R$ when $0\le s< 1$.

\item the function $\Omega_1(x)-\log (1/\sqrt{\vert x\vert})$ is bounded on $\mathbb  R$.

\item the function $\Omega_s(x)$ is differentiable on $\mathbb {R}\setminus\{0\}$ and continuous at $x=0$ for any $s>1$.

 \end{itemize}
Given the definition of $\Omega_s(x)$ by means of the function $I_s(x)$ (see~\eqref{eq:8}), 
it is enough to prove these facts for $x\ge 0$. 

\subsection{Continuity and differentiability of $\Omega_s(x)$ on $(0,+\infty)$}
\

Note that $V_s(x,0)=0$ and thus
$
I_s(x) = U_s(x)+W_s(x)
$
where 
$$
U_{s}(x)= \int\limits_{1/2-\rho \infty}^{1/2+\rho \infty} \frac{e^{i\pi z^2 x}}{z^s(1-e^{2i\pi z})} \dd z
$$
and 
$$
W_{s}(x)=\rho x^s\int\limits_{-\infty}^{\infty} e^{-\pi x u^2}
\left(\sum_{k=1}^{\infty} e^{-i\pi k^2/x} 
\bigg(\frac{1}{(\rho xu+k)^s}-\frac{1}{k^s}\bigg) \right) \dd u.
$$
It is clear that $U_{s}(x)$ defines a differentiable function on 
$(0,+\infty)$. The 
series 
$$
\widehat{W}_{s}(x,u,0)=\sum_{k=1}^{\infty} e^{-i\pi k^2/x} 
\bigg(\frac{1}{(\rho xu+k)^s}-\frac{1}{k^s}\bigg)
$$
defines a continuous function of $(u,x)$ on $\mathbb R \times (0,+\infty)$. 
Moreover, by Lemma~\ref{lem:riv2} applied with $N=0$, we have
$\vert \widehat{W}_{s}(x,u,0)\vert \ll \vert u x \vert .$ 
This guarantees the continuity $W_{s}(x)$ on $[0,+\infty)$, and 
that $W_{s}(x)=\mathcal{O}(x^{s})$ on $[0,+\infty)$. Hence, $\Omega_s(x)$ is continuous 
on $(0,+\infty)$.

\medskip

To prove that $\Omega_s(x)$ is differentiable at any rational number $x=p/q$ with $p,q$ both odd, 
it remains to prove that this is the case of $W_s(x)$. Integrating by parts, we get
\begin{align*}
W_s(x)&=\rho x^s\int\limits_{-\infty}^{\infty} ue^{-\pi x u^2}
\left(\sum_{k=1}^{\infty} e^{-i\pi k^2/x} 
\frac1u\cdot \bigg(\frac{1}{(\rho xu+k)^s}-\frac{1}{k^s}\bigg) \right) \dd u
\\
&=\frac{x^{s-1}}{2\pi}\int\limits_{-\infty}^{\infty} e^{-\pi x u^2}
\left(\sum_{k=1}^{\infty} e^{-i\pi k^2/x} \frac{\dd}{\dd u}\bigg(
\frac1u \cdot\Big(\frac{1}{(\rho xu+k)^s}-\frac{1}{k^s}\Big)\bigg) \right) \dd u,
\end{align*}
where differentiation under the sum is allowed by uniform convergence of 
$\widehat{W}_{s}(x,u,0)$. We observe that
\begin{align*}
\frac{\dd}{\dd u}\bigg(
\frac1u \cdot\Big(\frac{1}{(\rho xu+k)^s}-\frac{1}{k^s}\Big)\bigg) 
&= \frac{-s\rho x}{u(k+\rho xu)^{s+1}}-\frac{1}{u^2(k+\rho xu)^s}+\frac{1}{u^2k^s}
\\
&= \frac{c_s(\rho x)^2}{k^{s+2}}+\mathcal{O}\left(\frac{1}{k^{s+3}}\right)
\end{align*}
for some constant $c_s>0$ independent of $k$ and $x$.
Therefore, 
\begin{multline}
W_s(x)=\frac{x^{s-1}}{2\pi}\int\limits_{-\infty}^{\infty} e^{-\pi x u^2}
\left(\sum_{k=1}^{\infty} e^{-i\pi k^2/x} \bigg(\frac{\dd}{\dd u}\bigg(
\frac1u \cdot\Big(\frac{1}{(\rho xu+k)^s}-\frac{1}{k^s}\Big)\bigg) -\frac{c_s(\rho x)^2}{k^{s+2}}
\bigg) \right) \dd u
\\
+ \frac{c(s)\rho^3 x^{s+1}}{2\pi}  \int_{-\infty}^{\infty} e^{-\pi x u^2} \dd u\cdot  
\sum_{k=1}^{\infty} \frac{e^{-i\pi k^2/x}}{k^{s+2}} \label{riv1806}.
\end{multline}

On the one hand, the series 
$$
\sum_{k=1}^{\infty} e^{-i\pi k^2/x} \bigg(\frac{\dd}{\dd u}\bigg(
\frac1u \cdot\Big(\frac{1}{(\rho xu+k)^s}-\frac{1}{k^s}\Big)\bigg) -\frac{c_s(\rho x)^2}{k^{s+2}}
\bigg)
$$
can be termwise differentiated with respect to $x$ 
and it follows easily that the first term 
on the right hand side of~\eqref{riv1806} is differentiable for any $x>0$.

On the other hand, the second term on the right hand side of~\eqref{riv1806} is 
\begin{equation}
\label{eqluther}
\frac{c_s\rho^3 x^{s+1/2}}{2\pi}  \cdot  
\sum_{k=1}^{\infty} \frac{e^{-i\pi k^2/x}}{k^{s+2}}. 
\end{equation}
A result of Luther~\cite{luther} ensures that the series 
$
\sum_{k=1}^{\infty} \frac{e^{i\pi k^2x}}{k^{\xi}}
$
is differentiable at any rational number $x=p/q$, $p,q$ odd, for any fixed $\xi>3/2$  
(and this is no longer true when $\xi\le 3/2$). Since $s\ge 0$, this result is more than needed to 
complete the proof that $W_s(x)$ is differentiable at any rational number $x=p/q$, $p,q$ odd.

Moreover, when $s>1$, the series~\eqref{eqluther} is everywhere differentiable (except at 0) by uniform convergence, hence $\Omega_s(x)$ is differentiable on $(0,+\infty)$.

\subsection{Local behavior of $\Omega_s(x)$ around $x=0$}
\

It was proved above  that $W_{s}(x)=\mathcal{O}(x^{s})$ on $[0,+\infty)$ and $U_{s}(x)$ is continuous on $(0,+\infty)$.
We will now establish that 
\begin{itemize} 
\item  $U_{s}(x)$ is continuous at $x=0$ for any $s>1$.
\item  $U_{s}(x)=\frac{\rho^{1-s}\Gamma(\frac{1-s}{2})}{2\pi^{\frac{1-s}{2}}}
x^{\frac{s-1}{2}}+\mathcal{O}(1)$ on $(0,+\infty)$ for any $0\le s< 1$.
\item $U_{1}(x)= \log(1/\sqrt{x})+\mathcal{O}(1)$ on $(0,+\infty)$.
\end{itemize}

The case $s>1$ is easy because 
$$
\lim_{x\to 0^+} U_s(x)=U_s(0)= \int\limits_{1/2-\rho \infty}^{1/2+\rho \infty} 
\frac{1}{z^s(1-e^{2i\pi z})} \dd z
$$
is finite, by Lebesgue dominated convergence theorem. This implies that $\Omega_s(x)$ is continuous at $x=0$ when $s>1$.

We assume from now on that $s\in[0,1]$. 
The change of variable $z=\frac12+\rho u$ gives
$$
U_s(x)=\rho e^{i\pi x/4}\int\limits_{-\infty}^{\infty}\frac{e^{-\pi x u^2}e^{i\rho \pi x u}}
{(\frac12+\rho u)^s(1+e^{2i\rho \pi u})} \dd u.
$$

If $u\le 0$, then since $e^{2i\rho {\pi} u}=e^{i\sqrt{2\pi}u}e^{-\sqrt{2\pi}u}$ we have
$$
\left\vert g(u) :=\frac{1}{1+e^{2i\rho \pi u }}=\frac{e^{-i\sqrt{2}\pi u}e^{\sqrt{2}\pi u}}
{1+e^{-i\sqrt{2}\pi u}e^{\sqrt{2}\pi u}}  \right\vert
\ll e^{\pi \sqrt{2} u}
$$
and thus 
$$
\left\vert\frac{e^{-\pi x u^2}e^{i\rho \pi x u}}
{(\frac12+\rho u)^s(1+e^{2i\rho \pi u})} \right\vert \ll 
\frac{e^{\sqrt{2}\pi u}}{\vert \frac12 +\rho u\vert^s}.
$$
Since the integral
$$
\int\limits_{-\infty}^{0}\frac{e^{\sqrt{2}\pi u}}{\vert \frac12 +\rho u\vert^s} \,\dd u
$$
is convergent, by Lebesgue dominated convergence theorem, we get that 
$$
\lim_{x\to 0^+}  
\int\limits_{1/2-\rho \infty}^{1/2} \frac{e^{i\pi z^2 x}}{z^s(e^{2i\pi z}-1)} 
\dd z = \int\limits_{1/2-\rho \infty}^{1/2} \frac{\dd z}{z^s(e^{2i\pi z}-1)} =:\alpha_s.
$$

If $u\ge 0$, $g(u)\to 1$ when $u\to +\infty$ and we cannot invoke Lebesgue's theorem 
because $1/\vert \frac12 +\rho u\vert^s$ is not integrable over $[0,+\infty)$. However, 
$$
g(u)-1=-\frac{e^{2i\pi \rho u}}{1+e^{2i\pi \rho u}}
$$
so that $\vert g(u)-1\vert \ll e^{-\sqrt{2}\pi u}$ on $[0,+\infty)$, and Lebesgue's 
theorem entails that
$$
\lim_{x\to 0^+}  
\int\limits_{1/2}^{1/2+\rho \infty} 
\frac{e^{i\pi z^2 x}}{z^s}\left(\frac1{e^{2i\pi z}-1}-1\right) \dd z = 
\int\limits_{1/2}^{1/2+\rho \infty} 
\frac{1}{z^s}\left(\frac1{e^{2i\pi z}-1}-1\right) \dd z =:\beta_s.
$$
It thus remains to study the simpler integral 
$$
P_s(x):=\int\limits_{1/2}^{1/2+\rho \infty} 
\frac{e^{i\pi z^2 x}}{z^s} \,\dd z
$$
because 
$
U_s(x)=P_s(x)+\alpha_s+\beta_s+o(1)
$
as $x\to 0^+$. Setting $z=\frac12+\rho v/\sqrt{w}$, we obtain
$$
P_s(x)= 
\rho^{1-s}e^{i\pi x/4}x^{\frac{s-1}{2}}\int\limits_0^{+\infty} e^{-\pi v^2}
\frac{e^{i\rho \pi \sqrt{x}v}}{(v+\rhob\frac{\sqrt{x}}{2})^s} \dd v.
$$
We see here that as $x\to 0^+$, the integral is ``close'' to the integral 
$\int\limits_0^{+\infty} e^{-\pi v^2} \dd v/v^s$. The latter is convergent if $0\le s<1$ 
but divergent if $s=1$ and we now have to distinguish between both possibilities.

\medskip

\noindent {\bf Case $0\le s<1$}. 
Observe that  $\vert \exp(i\rr \pi \sqrt{x}v)-1\vert \ll \sqrt{x}v$ for $v\in [0,+\infty)$ and 
any $x\ge 0$. (this is true when $\sqrt{x}v$ by using the Taylor polynomial, 
and it goes to zero when $\sqrt{x}v$ tends to $+\infty$.)
Moreover, $\vert v+\rhob\sqrt{x}/2\vert\ge \vert v+\sqrt{2x}\vert$. Hence, 
$$
\left\vert \int\limits_0^{+\infty} e^{-\pi v^2}\frac{e^{i\rho \pi \sqrt{x}v}-1}
{(v+\frac{\rhob\sqrt{x}}{2})^s} \,\dd v \right\vert \ll \int\limits_0^{+\infty} e^{-\pi v^2}
\frac{\sqrt{x}v}
{(v+\sqrt{2x})^s} \dd v \le \sqrt{x} \int\limits_0^{+\infty} v^{1-s}e^{-\pi v^2} =\mathcal{O}(\sqrt{x}).
$$
Therefore, 
\begin{equation}\label{riv:180}
P_s(x) = \rho^{1-s} x^{\frac{s-1}{2}}\int\limits_0^{+\infty} 
\frac{e^{-\pi v^2}}{(v+\rhob \sqrt{x}/2)^s} \dd v + \mathcal{O}(x^{\frac s2}).
\end{equation}
We now prove that
\begin{equation}\label{riv:181}
\int\limits_0^{+\infty} 
\frac{e^{-\pi v^2}}{(v+\rhob \sqrt{x}/2)^s} \dd v = \int\limits_0^{+\infty} 
\frac{e^{-\pi v^2}}{v^s} \dd v + \mathcal{O}(x^{\frac{1-s}2}). 
\end{equation}
Indeed, with $X=\rhob \sqrt{x}/2$, we have
\begin{equation*}
\int\limits_0^{+\infty} 
e^{-\pi v^2}\left(\frac{1}{(v+X)^s} -\frac1{v^s}\right) \dd v 
\ll \vert X\vert \int_{\vert X\vert}^{\infty}\frac{e^{-v}}{v^{s+1}} 
\dd v + \int_0^{\vert X\vert} \frac{e^{-v}}{v^s} \dd v 
\ll X^{1-s}, 
\end{equation*}
where we have used that 
$$
\left\vert \frac{1}{(v+X)^s} -\frac1{v^s}\right\vert \ll 
\begin{cases}
\frac{\vert X\vert}{v^{s+1}} \quad \textup{if} \quad \vert X\vert \le v
\\
\frac{1}{v^{s}} \quad \textup{if} \quad \vert X\vert \ge v.
\end{cases}
$$
Using~\eqref{riv:181} in~\eqref{riv:180} gives
$$
P_s(x)
= \frac{\rho^{1-s}\Gamma(\frac{1-s}{2})}{2\pi^{\frac{1-s}{2}}} x^{\frac{s-1}{2}} + \mathcal{O}(1)
$$
and finally
$$
U_s(x)=\frac{\rho^{1-s}\Gamma(\frac{1-s}{2})}{2\pi^{\frac{1-s}{2}}} x^{\frac{s-1}{2}} + \mathcal{O}(1)
$$
as $x\to 0^+$.

\medskip

\noindent {\bf Case $s=1$}. We start in a similar way:  
$$
\left\vert \int\limits_0^{+\infty} e^{-\pi v^2}\frac{e^{i\rho \pi \sqrt{x}v}-1}
{\frac{\sqrt{x}}{2}+\rho v} \,\dd v \right\vert \ll \int\limits_0^{+\infty} e^{-\pi v^2}
\frac{\sqrt{x}v}
{v+\sqrt{2x}} \dd v \le \sqrt{x}\int\limits_0^{+\infty} e^{-\pi v^2} \ll \sqrt{x}.
$$
Therefore
$$
P_1(x)=e^{i\pi x/4}\int\limits_0^{+\infty} 
\frac{e^{-\pi v^2}}{v+\rhob\frac{\sqrt{x}}{2}} \dd v+ \mathcal{O}( \sqrt{x}).
$$
Integrating by parts, we 
get 
\begin{align*}
\int\limits_0^{+\infty} 
\frac{e^{-\pi v^2}}{v+\rhob\frac{\sqrt{x}}{2}} \,\dd v &
=\left[ \log\Big(v+\rhob\frac{\sqrt{x}}{2}\Big)e^{-\pi v^2}\right]_0^{+\infty}+2\pi 
\int\limits_0^{+\infty} v \log\Big(v+\rhob\frac{\sqrt{x}}{2}\Big)e^{-\pi v^2}\dd v \notag
\\
&= \log(1/\sqrt{x})+  \mathcal{O}(1)  .
\end{align*}
Thus, 
$P_1(x)=\log(1/\sqrt{x}) + \mathcal{O}(1)$ when $x\to 0^+$ and the same estimate holds 
for $U_1(x)$ as claimed  in item 2.

\section{Proof of Theorem~\ref{theo:2}}\label{sec:4}

For simplicity, we set $r(x)=\exp(i\pi \si(x)/4)$. 
For any integer $n\ge0$, we define an integer $K(\ell,n)$ as follows: $K(-1,n)=n$ and 
$$
K(\ell,n)=\lfloor \lfloor \cdots \lfloor\lfloor n \vert x\vert \rfloor \vert T(x)\vert 
\rfloor\cdots \rfloor \vert T^{\ell}(x)\vert \rfloor
$$
for any integer $\ell\ge 0$. For instance, $K(0,n)=\lfloor n \vert x\vert \rfloor$ 
and  $K(1,n)=\lfloor\lfloor n \vert x\vert \rfloor \vert T(x)\vert\rfloor$.

We also set $L(n)=\min\{j\ge 0 : K(j,n)=0\}$. This integer is well defined. Indeed, 
by Proposition \ref{prop:ecf2},  it is obvious that 
$$
0\le K(j,n) \le n \vert xT(x)\cdots T^{j}(x)\vert \le \frac{n}{q_{j+1}-q_j} 
$$
and the right hand side tends to $0$ as $j\to +\infty$. 
By definition, we have $K(L(n)-1,n)\ge 1$. It is clear that $\lim_n L(n)=+\infty$ because otherwise 
if $L(n)$ were bounded, we would have $\lim_n K(L(n),n)=+\infty$ which is false.

\subsection{Proof of Theorem~\ref{theo:2}, part $(i)$}

\

We write 
$\Omega_s(x)=c(s)\vert x\vert^{\frac{s-1}{2}}+\Delta_s(x)$ where 
$c(s)=\frac{\rho^{1-s}\Gamma(\frac{1-s}{2})}{2\pi^{\frac{1-s}{2}}}$ and 
the function $\Delta_s(x)$ is 
bounded on $[-1,1]$ by Theorem~\ref{theo:1}$(ii)$. 
Let us define $E_s(n,x)$ as the error term in~\eqref{eq:11}, which is such that  
$$
E_s(n,x)=\mathcal{O}\bigg(\min\Big(\frac{1}{(n+1)^s\sqrt{\vert x\vert}}, 
\frac{1}{\vert x\vert^{\frac{1-s}{2}}}\Big)\bigg)
$$
For any irrational number $x\in (-1,1)$, Eq.~\eqref{eq:11} reads
\begin{equation}\label{eq:101}
F_{s,n}(x)=r(x)\vert x\vert^{s-\frac12}F_{s,\lfloor n x \rfloor}\big(T(x)\big) 
+c(s)\vert x\vert^{\frac{s-1}{2}}+\Delta_s(x)+E_s(n,x).
\end{equation}
Since $T(x)$ is also an irrational number in $(-1,1)$, we can use~\eqref{eq:101} 
with $x$ and $n$ replaced by $T(x)$ and $\lfloor n x \rfloor$ respectively, 
and iterate again the result. Formally, we find that for any integer $L\ge 0$,
\begin{align}
F_{s,n}(x)&=c(s)\sum_{j=0}^L r(x)r(T(x))\cdots r(T^{j-1}(x))
\frac{\vert xT(x)\cdots T^{j-1}(x)\vert^{s-\frac12}}{\vert T^{j}(x)\vert^{\frac{1-s}{2}}} \notag
\\
&+\sum_{j=0}^L r(x)r(T(x))\cdots r(T^{j-1}(x))
\vert xT(x)\cdots T^{j-1}(x)\vert^{s-\frac12}\, \Delta_s\big(T^{j}(x)\big)\notag
\\
&+r(x)r(T(x))\cdots r(T^{L}(x))
\vert xT(x)\cdots T^{L}(x)\vert^{s-\frac12}\, F_{s,K(L,n)}\big(T^{L+1}(x)\big)\notag
\\
&+\sum_{j=0}^L r(x)r(T(x))\cdots r(T^{j-1}(x))
\vert xT(x)\cdots T^{j-1}(x)\vert^{s-\frac12}\, E_s(K(j-1,n), T^j(x)).\label{eq:121}
\end{align}
With $L=L(n)$,~\eqref{eq:121} becomes 
\begin{align}
F_{s,n}(x)&=c(s)\sum_{j=0}^{L(n)} r(x)r(T(x))\cdots r(T^{j-1}(x))
\frac{\vert xT(x)\cdots T^{j-1}(x)\vert^{s-\frac12}}{\vert T^{j}(x)\vert^{\frac{1-s}{2}}} \notag
\\
&+\sum_{j=0}^{L(n)} r(x)r(T(x))\cdots r(T^{j-1}(x))
\vert xT(x)\cdots T^{j-1}(x)\vert^{s-\frac12}\, \Delta_s\big(T^{j}(x)\big)\notag
\\
&+\sum_{j=0}^{L(n)}r(x)r(T(x))\cdots r(T^{j-1}(x))
\vert xT(x)\cdots T^{j-1}(x)\vert^{s-\frac12}\, E_s(K(j-1,n), T^j(x)).\label{eq:161}
\end{align}
because $F_{s,K(L(n),n)}\big(T^{L(n)+1}(x)\big)=0$, being an empty sum. 

\medskip

Under hypothesis~\eqref{eq:riv10} of Theorem~\ref{theo:2}$(i)$, the series 
$$
c(s)\sum_{j=0}^{\infty} r(x)r(T(x))\cdots r(T^{j-1}(x))
\frac{\vert xT(x)\cdots T^{j-1}(x)\vert^{s-\frac12}}{ \vert T^j(x)\vert^{\frac{1-s}{2}}}
$$
converge absolutely (recall that $\big\vert r(x)r(T(x))\cdots r(T^{j-1}(x))\big\vert=1$) and 
this also forces the absolute convergence of the series
$$
\sum_{j=0}^{\infty} r(x)r(T(x))\cdots r(T^{j-1}(x))
\vert xT(x)\cdots T^{j-1}(x)\vert^{s-\frac12}\, \Delta_s\big(T^{j}(x)\big)
$$
because $\vert T^j(x)\vert\le 1$ and $\Delta_s$ is bounded on $[-1,1]$. 
Hence, since $L(n)\to +\infty$ with $n$ 
and by~\eqref{eq:121}, 
the convergence of $F_{s,n}(x)$ will follow from the proof that
$$
\sum_{j=0}^{L(n)}r(x)r(T(x))\cdots r(T^{j-1}(x)) 
\vert xT(x)\cdots T^{j-1}(x)\vert^{s-\frac12}\, E_s(K(j-1,n), T^j(x))
$$
tends to $0$ as $n\to +\infty$. For this, we observe that (by Theorem~\ref{theo:1})
$$
\lim_{n\to +\infty} E_s(K(j-1,n), T^j(x)) =0,
$$
\begin{equation}\label{eq:riv2006}
\vert E_s(K(j-1,n), T^j(x))\vert \ll_s \frac{1}{\vert T^j(x)\vert^{\frac{1-s}{2}}}
\end{equation}
and
$$
\sum_{j=0}^{\infty} \frac{\vert xT(x)\cdots T^{j-1}(x)\vert^{s-\frac12}}
{\vert T^j(x)\vert^{\frac{1-s}{2}}}<+\infty
$$
by hypothesis~\eqref{eq:riv10}. Hence, Tannery's theorem can be applied and it yields that
$$
\lim_{n\to+\infty}\sum_{j=0}^{L(n)}r(x)r(T(x))\cdots r(T^{j-1}(x))
\vert xT(x)\cdots T^{j-1}(x)\vert^{s-\frac12}\, E_s(K(j-1,n), T^j(x))=0
$$
as expected. We now let $n\to+\infty$ in~\eqref{eq:161} 
to get the identity
$$
F_s(x)=\sum_{j=0}^{\infty} r(x)r(T(x))\cdots r(T^{j-1}(x))
\vert xT(x)\cdots T^{j-1}(x)\vert^{s-\frac12}\, \Omega_s\big(T^{j}(x)\big).
$$

\subsection{Proof of Theorem~\ref{theo:2}, part $(ii)$}\label{ssec:theo2ii} 

\

The proof is more complicated than when $s<1$. We can write a 
bound similar to~\eqref{eq:riv2006} but the corresponding right hand side depends on $n$ and 
we cannot invoke Tannery's theorem. Instead we use an ad hoc inelegant method.

For this, we first need more informations on $K(j-1,n)$. 
By multiple applications of the trivial inequality $\lfloor \al \rfloor \ge \al-1$, we get
\begin{align}
K(j-1,n) &\ge n\vert xT(x)\cdots T^{j-1}(x)\vert - \sum_{k=1}^j \vert T^k(x)\cdots T^{j-1}(x)\vert 
\notag
\\
&=n\vert xT(x)\cdots T^{j-1}(x)\vert
\cdot\left(1-\frac1n\sum_{k=1}^j \frac1{\vert xT(x)\cdots T^{k-1}(x)\vert} \right). \label{eq:13}
\end{align}
We define $J=J(n)$ as the maximal integer such that $Jq_J\le n/4$; it 
is clear that $J \to +\infty$ with $n$. We claim that
$$
\sum_{k=1}^j \frac1{\vert xT(x)\cdots T^{k-1}(x)\vert} \le \frac{n}{2}
$$ 
for all $j\in\{ 0, \ldots, J\}$. Indeed, the sequence 
$\vert xT(x)\cdots T^{k-1}(x)\vert$ is non-increasing 
(because $\vert T^m(x)\vert \le 1$) and thus
$$
\sum_{k=1}^j \frac1{\vert xT(x)\cdots T^{k-1}(x)\vert} 
\le \frac{j}{\vert xT(x)\cdots T^{j-1}(x)\vert} 
=j\vert q_j+e_j x_j q_{j-1}\vert\le 2jq_j\le 2Jq_J \le \frac{n}{2}
$$ 
(Note that the sequence $(jq_j)_j$ is increasing.) 
Therefore,~\eqref{eq:13} yields that 
$$
K(j-1,n) \ge \frac{n}{2}\vert xT(x)\cdots T^{j-1}(x)\vert
$$
for any $j\in\{0,\ldots, J\}$. Moreover, for those $j$, 
we also have $K(j-1,n)\ge 1$. 

We write 
$\Omega_1(x)=\log(1/\sqrt{\vert x\vert})+\Delta_1(x)$ where the function $\Delta_1(x)$ is 
bounded on $[-1,1]$ by Theorem~\ref{theo:1}$(ii)$. 
Let us define $E_1(n,x)$ as the error term in~\eqref{eq:11} for $s=1$, i.e. 
$$
E_1(n,x)=\mathcal{O}\bigg(\min\Big(\frac{1}{(n+1)\sqrt{\vert x\vert }}, 
\big\vert \log( (n+1)\sqrt{\vert x\vert })\big\vert
+1\Big)\bigg)
$$
For any irrational number $x\in (-1,1)$, Eq.~\eqref{eq:11} reads
\begin{equation}\label{eq:10}
F_{1,n}(x)=r(x)\sqrt{\vert x\vert}F_{1,\lfloor n x \rfloor}\big(T(x)\big) 
-\frac12\log \vert x\vert+\Delta_1(x)+E_1(n,x).
\end{equation}
Since $T(x)$ is also an irrational number in $(-1,1)$, we can use~\eqref{eq:10} 
with $x$ and $n$ replaced by $T(x)$ and $\lfloor n x \rfloor$ respectively, 
and iterate again the result. Formally, we find that for any integer $L\ge 0$,
\begin{align}
F_{1,n}(x)&=-\frac12\sum_{j=0}^L r(x)r(T(x))\cdots r(T^{j-1}(x))
\sqrt{\vert xT(x)\cdots T^{j-1}(x)\vert} \,\log\vert T^{j}(x)\vert \notag
\\
&+\sum_{j=0}^L r(x)r(T(x))\cdots r(T^{j-1}(x))
\sqrt{\vert xT(x)\cdots T^{j-1}(x)\vert}\, \Delta_1\big(T^{j}(x)\big)\notag
\\
&+r(x)r(T(x))\cdots r(T^{L}(x))
\sqrt{\vert xT(x)\cdots T^{L}(x)\vert}\, F_{1,K(L,n)}\big(T^{L+1}(x)\big)\notag
\\
&+\sum_{j=0}^L r(x)r(T(x))\cdots r(T^{j-1}(x))
 \sqrt{\vert xT(x)\cdots T^{j-1}(x)\vert}\, E_1(K(j-1,n), T^j(x)).\label{eq:12}
\end{align}
From now on, the letter $L$ will stand exclusively for $L(n)$ (defined above), 
in which case~\eqref{eq:12} becomes 
\begin{align}
F_{1,n}(x)&=-\frac12\sum_{j=0}^L r(x)r(T(x))\cdots r(T^{j-1}(x))
\sqrt{\vert xT(x)\cdots T^{j-1}(x)\vert} \,\log\vert T^{j}(x)\vert\notag
\\
&+\sum_{j=0}^L r(x)r(T(x))\cdots r(T^{j-1}(x))
\sqrt{\vert xT(x)\cdots T^{j-1}(x)\vert}\, \Delta_1\big(T^{j}(x)\big)\notag
\\
&+\sum_{j=0}^L r(x)r(T(x))\cdots r(T^{j-1}(x))
\sqrt{\vert xT(x)\cdots T^{j-1}(x)\vert}\, E_1(K(j-1,n), T^j(x)).\label{eq:16}
\end{align}
because $F_{1,K(L,n)}\big(T^{L+1}(x)\big)=0$, being an empty sum. 

\medskip

Under hypothesis~\eqref{eq:3} of 
Theorem~\ref{theo:2}$(ii)$, the series 
$$
-\frac12 \sum_{j = 0}^{\infty} r(x)r(T(x))\cdots r(T^{j-1}(x))
\sqrt{\vert xT(x)\cdots T^{j-1}(x)\vert} 
\,\log\vert T^{j}(x)\vert,
$$
converges absolutely and hypothesis~\eqref{eq:33} implies the absolute convergence
$$
\sum_{j= 0}^{\infty} r(x)r(T(x))\cdots r(T^{j-1}(x))
\sqrt{\vert xT(x)\cdots T^{j-1}(x)\vert}, 
$$
which in turn forces the absolute convergence of 
$$
\sum_{j= 0}^{\infty} r(x)r(T(x))\cdots r(T^{j-1}(x))
\sqrt{\vert xT(x)\cdots T^{j-1}(x)\vert}\, \Delta_1\big(T^{j}(x)\big)
$$
by boundedness of $\Delta_1$ on $[-1,1]$. Hence, since $L(n)\to +\infty$ with $n$ 
and by~\eqref{eq:12}, 
the convergence of $F_{1,n}(x)$ will follow from the proof that
$$
R_1(n,x) :=\sum_{j=0}^{L(n)}\sqrt{\vert xT(x)\cdots T^{j-1}(x)\vert}\, E_1(K(j-1,n), T^j(x))
$$
tends to $0$ as $n\to +\infty$. 

With $H=\min(L,J)$, we split $R_1(n,x)$ into three parts:
\begin{multline*}
R_1(n,x) = \left(\sum_{j=0}^{H-1}+\sum_{j=H}^{L-1}\right) 
\sqrt{\vert xT(x)\cdots T^{j-1}(x)\vert}\, E_1(K(j-1,n), T^j(x)) 
\\
+ \sqrt{\vert xT(x)\cdots T^{L-1}(x)\vert}\, E_1(K(L-1,n), T^L(x)).
\end{multline*}
The sum from $H$ to $L-1$ might be empty if  $H=L$, 
which would simplify the proof below. We assume that it is not empty, 
i.e that $H=J\le L-1$. 
For the two sums from $j=0$ to $j=L-1$, we use the fact that (for any $n\ge 0$, any $x>0$)
$$
E_1(n,x) = \mathcal{O}\left(\frac{1}{(n+1)\sqrt{\vert x\vert}}\right).
$$
First
\begin{align*}
\sum_{j=0}^{H-1}&\sqrt{\vert xT(x)\cdots T^{j-1}(x)\vert}\, E_1(K(j-1,n), T^j(x)) \\
&\ll \sum_{j=0}^{J-1} \frac{\sqrt{\vert xT(x)\cdots T^{j-1}(x)\vert}}{(K(j-1,n)+1) 
\sqrt{\vert T^j(x) \vert} }
\le \frac2n\sum_{j=0}^{J-1} 
\frac{\sqrt{\vert xT(x)\cdots T^{j-1}(x)\vert}}
{(\vert xT(x)\cdots T^{j-1}(x) \vert +1)\sqrt{\vert T^j(x) \vert}}
\\
&\le \frac2n\sum_{j=0}^{J-1} 
\frac{1}{\sqrt{\vert xT(x)\cdots T^j(x) \vert}}
\le \frac{2J}{n\sqrt{\vert xT(x)\cdots  T^{J-1}(x) \vert}}
\\
&\le \frac{4J\sqrt{q_{J}}}{n} \le \frac{4\sqrt{Jq_J}}{n}\le \frac{2}{\sqrt{n}}
\end{align*}
by definition of $J$. 

Second, for the sum from $j=H$ to $j=L-1$, we observe 
that $K(j-1,n)\ge \frac{1}{\vert T^j(x)\vert}$ because 
$1\le K(j,n)=\lfloor K(j-1,n) \vert T^j(x)\vert \rfloor$ for such $j$ 
by definition of $L$. Hence, 
\begin{align}
\sum_{j=H}^{L-1}&\sqrt{\vert xT(x)\cdots T^{j-1}(x)\vert}\, E_1(K(j-1,n), 
T^j(x))\notag
\\
&\ll \sum_{j=H}^{L-1} \frac{\sqrt{\vert xT(x)\cdots T^{j-1}(x)\vert}}{(K(j-1,n)+1)\sqrt{\vert T^j(x)\vert }}
\le \sum_{j=H}^{L-1} \frac{\sqrt{\vert xT(x)\cdots T^{j-1}(x)\vert} \cdot \vert T^j(x)\vert}
{ \sqrt{\vert T^j(x) \vert}} \notag
\\
&\le  \sum_{j=H}^{L-1} \sqrt{\vert xT(x)\cdots T^{j}(x)\vert}. \label{eq:14}
\end{align}
The series with term $\sqrt{\vert xT(x)\cdots T^{j}(x)\vert}$ is convergent (by hypothesis), 
thus the expression in~\eqref{eq:14} tends to $0$ as $n$ tends to infinity because 
$H$ tends to $+\infty$ with $n$. 
It remains to consider the case for $j=L$. Here, we use the fact that (for any $n\ge 0$, any $x>0$)
$$
E_1(n,x) \ll  \big\vert \log ((n+1) \sqrt{\vert x\vert})\big\vert+1.
$$
(In this case, $n$ and $x$ 
are replaced by $K(L-1,n)$ and $T^L(x)$ respectively.)  
Hence, 
\begin{multline}\label{eq:15}
\sqrt{\vert xT(x)\cdots T^{L-1}(x)\vert}\, E(K(L-1,n), T^L(x))
\\
\ll 
\sqrt{\vert xT(x)\cdots T^{L-1}(x)\vert}\cdot 
\Big(1+\big\vert \log \big((K(L-1,n)+1) \sqrt{\vert T^L(x)}\big)
\big\vert\Big)
\end{multline}
The properties $K(L-1,n)\ge 1$ and $\lfloor K(L-1,n)\vert T^L(x)\vert \rfloor=0$ together imply 
that $\vert T^L(x) \vert\le K(L-1,n)\vert T^L(x)\vert < 1$, so that 
\begin{align*}
\big\vert \log \big((K(L-1,n)+1) \sqrt{\vert T^L(x)\vert }\big)\big\vert 
&\le 
\big\vert \log \big(K(L-1,n) \sqrt{\vert T^L(x)\vert }\big)\big\vert + \log(2)
\\
& \le \frac32 \big\vert \log \vert T^L(x)\vert\big\vert+\log(2). 
\end{align*}
Therefore,~\eqref{eq:15} becomes
\begin{multline}\label{eq:riv12}
\sqrt{\vert xT(x)\cdots T^{L-1}(x)\vert}\, E(K(L-1,n), T^L(x))
\\
\ll 
\sqrt{\vert xT(x)\cdots T^{L-1}(x)\vert}\cdot\big(1+ \big\vert \log\vert T^L(x)\vert \big\vert\big). 
\end{multline}
Hypothesis~\eqref{eq:33} and~\eqref{eq:3} in Theorem~\ref{theo:2}$(ii)$ now ensure that 
the right hand side of~\eqref{eq:riv12} tends to $0$ when 
$L(n)\to +\infty$. 
This concludes the proof that $R_1(n,x)$ tends to $0$ when $n\to +\infty$. 

We now let $n\to+\infty$ in~\eqref{eq:16} 
to get the identity
$$
F_1(x)=\sum_{j=0}^{\infty} r(x)r(T(x))\cdots r(T^{j-1}(x))
\sqrt{\vert xT(x)\cdots T^{j-1}(x)\vert}\, \Omega_1\big(T^{j}(x)\big).
$$

\section{Proof of Theorem~\ref{theo:5}  }\label{sec:10_2}
 
\subsection{Proof of Theorem~\ref{theo:5}, part $(i)$}
 
\

We introduce a second dynamical system more adapted to our study 
(see the work of Schweiger~\cite{S1,S2}). 
Let us partition the interval $(0,1]$ into the double-indexed sequence of intervals 
$$ B(+1, k)= \left(\frac{1}{2k},\frac{1}{2k-1} \right] \ 
\mbox{ and } \ B(-1,k)=  \left(\frac{1}{2k+1},\frac{1}{2k} \right] .$$
Consider the map $U: \zu \to \zu $ by $U(0)=0$ and if $x\neq 0$
$$U(x) = e\cdot\left(\frac{1}{x}-2k  \right) \ \mbox { when } x\in B(e,k).$$
It is trivial to check that for every irrational $x\in \zu$, 
$|U(x)|=|T(x)|$. So we will prove Theorem~\ref{theo:5} using this 
transformation $U$ instead of $T$.

A key property is that $U$ is an ergodic transformation with a 
$\sigma$-finite invariant measure $\mu$ with infinite mass (due to 
the parabolic point $1$ for the mapping $U$), whose 
density relatively to the Lebesgue measure is given by
$$
d_\mu(x) = \frac{1}{x+1}+\frac{1}{1-x}.
$$
 
Further we need to study in details the orbit of a typical point $x$ 
under the action of the dynamical system $(\zu,U)$. For this, 
we introduce the points $x_p := \frac{p-1}{p} = 1-\frac{1}{p}$, 
for every integer $p\geq 1$. Observe that for every $p\geq 2$,
\begin{equation}
\label{uxp}
U(x_p) =x_{p-1}.
\end{equation}

\begin{lem}
\label{lem:33}
Let $p\geq 2$ be an integer, and let $y$ be such that 
$$
x_{p-1} \leq   y  \leq x_{p}.
$$
Then for every $ 0 \leq m \leq p-2  $, one has 
\begin{align}
\label{seuret:1}
x_{p-1-m}  \ \leq  &    U^{m}(y)   <    \  x_{p-m}
\\
\nonumber
\frac{p-m-2}{p-1}  \  \leq  & yU(y) \cdots U^{ m}(y) \leq  \   \frac{p-m-1}{p}.
\end{align}
\end{lem}
\begin{proof}
Equation~\eqref{seuret:1} follows from~\eqref{uxp} and the 
monotonicity of $U$ on the interval $[1/2,1)$. Further,  
\begin{align*}
U^{P_k+p^1_k}(x)\cdots  U^{P_k+p^1_k+m}(x)  & \leq   x_{p }  x_{p-1} \cdots  x_{p-m}
\\
&  \leq  \frac{p-1}{p}\frac{p-2}{p-1} \cdots \frac{p-m-1}{p-m} 
= \frac{16}{p} = \frac{p-m-1}{p}.
\end{align*} 
The same holds for the lower bound.
\end{proof}

Recall that $\sigma(x)$ stands for the sign of $x\in \R$.

\begin{lem}
\label{lem:22}
For every irrational $x\in \zu$, there are two possibilities: either 
there exists an integer $j_x$ such that for every 
$j\geq j_x$, $U^j(x) \leq x_{2}$, or  there exist two sequences of 
integers $(p^1_k)_{k\geq 1}$ and $(p^2_k)_{k\geq 1}$ satisfying the 
following:  for every $k\geq 4$, $p^1_k\geq 1$, $p^2_k\geq 1$, 
and if we set $P_{k+1}= \sum_{\ell=1}^{ k} p^1_k+p^2_k$, then:
\begin{itemize}
\item for every $ P_k \leq j \leq P_k+p^1_{k} -1$, $ U^j(x) \leq  x_{2 } = 1/2$,
 
\item
for every $ P_k + p^1_k \leq j \leq P_k + p^1_k +p^k_2-1=P_{k+1} -1$, $  
x_{2} <U^{j+1}(x) < U^j(x)$, and  the $\sigma(T^j(x)) $  are all equal 
(i.e. the $T^j(x)$ have all the same sign).
\item
$p_k^2$ is the unique integer greater than 1 such that
$$ x_{p^k_2+1}\leq   U^{ P_k+p^1_{k} }(x) \leq x_{p^k_2+2 }$$
\end{itemize}
\end{lem}

\begin{proof}
Assume that we are not in the first case. Hence $U^j(x) > x_{2}$ for 
infinitely many integers $j$.

Consider the first integer $j$ such that $U^j(x) > x_{2}$, and call 
this integer $p^1_1$. Observe that it is possible that $p^1_1=0$, if $x$ 
itself is greater than $ x_{2}$. On the interval $[ x_{2},1]$,  as already 
stated in Lemma~\ref{lem:33}, the map $U$ is strictly decreasing and concave, 
and has a derivative strictly less than one  when $x<1$.  Consequently, as 
long as $U^j(x)$ stays in the interval $[ x_{2},1)$,  the sequence $U^j(x)$ 
is strictly decreasing. Moreover, using~\eqref{seuret:1},  there is 
a first integer $p^2_1$ such that $U^{p^1_1+p^2_1}(x)>   x_{2}$ and   
$U^{p^1_1+p^2_1}(x) \leq x_{2}$. 
 
Iterating this scheme allows to find the sequences $(p^1_k)$ and $(p^2_k)$. 
The fact that the $T^j(x)$ have all the same sign follows from the 
fact that $T(\pm[2/3,1]) \subset \pm [1/2,1)$.

The third item follows from the definition of $p_k^2$ and Lemma~\ref{lem:33}.
\end{proof}

The rest of this section is devoted to the proof of part $(i)$ of Theorem~\ref{theo:5}. 

We denote $M_\Omega$ a positive constant such that $|\Omega(x)| \leq M_\Omega$ for every $x$.
We fix an irrational number $x\in (0,1)$, and for convenience we will denote 
$u_j:=|U^j(x)|.$ 
If  there exists $j_x$ such that $u_j \leq x_{16}$ for every $j\geq j_x$, 
then the series converges. Thus, we assume that $u_j \geq x_{16}$ for 
infinitely many integers $j$.  Adapting Lemma~\ref{lem:22}, we   immediately get:
 
\begin{lem}
\label{lem:22bis}
For every irrational $x\in \zu$, there are two possibilities: either 
there exists an integer $j_x$ such that for every $j\geq j_x$, 
$U^j(x) \leq x_{16}$, or  there exist two sequences of integers 
$(p^1_k)_{k\geq 1}$ and $(p^2_k)_{k\geq 1}$ satisfying the following:    
for every $k\geq 4$, $p^1_k\geq 1$, $p^2_k\geq 14$, and if we set 
$P_{k+1}= \sum_{\ell=1}^{ k} p^1_k+p^2_k$, then:
\begin{itemize}
\item for every $ P_k \leq j \leq P_k+p^1_{k} -1$, $ U^j(x) \leq  x_{16}$,
 
\item
for every $ P_k + p^1_k \leq j \leq P_{k+1} -1$, $   x_{16} <U^{j+1}(x) < U^j(x)$.
\item
$p_k^2$ is the unique integer greater than 1 such that
$$ x_{p^k_2+15}\leq   U^{ P_k+p^1_{k} }(x) \leq x_{p^k_2+16 }$$
\end{itemize}
\end{lem}
 The difference with Lemma~\ref{lem:22} is that we impose 
$p_k^1 \geq 14$ for every $k$. This follows from the fact that 
the sequence $( U^{ m}(y))$ is slowly decreasing  when $y\in [ x_{2}, x_{16}]$.
 
It is clear that the convergence of the sequence~\eqref{seriesfinal} 
does not depend on the first terms. Hence, without loss of generality, 
we assume  that  $n_1\geq 4$. 
We now bound by above the partial sums 
$$
\Sigma_J= \sum_{j=1}^{J}e^{i\frac{\pi}4\sum\limits_{\ell=0}^{j-1}\si(T^\ell x)}
\vert xT(x)\cdots T^{j-1}(x)\vert^{\alpha} \,\Omega \big(T^j(x)\big).
$$

\noindent {\bf Step 1:} We separate the cases where $1\leq j \leq p^1_1-1$ 
and $n_1 \leq j \leq p^1_1+p^2_1-1=P_1-1$.

\smallskip

$\bullet$ 
If $j\leq p^1_1-1$, then  we  have
$$\sum_{j=1}^{p_1^1-1} \left|e^{i\frac{\pi}4\sum\limits_{\ell=0}^{j-1}\si(T^\ell x)}
\vert xT(x)\cdots T^{j-1}(x)\vert^{\alpha} \,\Omega \big(T^j(x)\big) \right| 
\leq M_\Omega \frac{x_{16} -x_{16} ^{p^1_1}}{1-x_{16} } \leq  M_\Omega \frac{x_{16} }{1-x_{16} } .$$
Hence
$$|\Sigma_{p^1_1-1}| \leq M_\Omega \frac{x_{16} }{1-x_{16} } .$$

\smallskip

$\bullet$ We now consider $ p^1_1\leq j  \leq p^1_1+p^2_1-1=P_1-1$. 
We observe that
\begin{multline} 
\Sigma_{P_1-1} -\Sigma_{p^1_1-1} =  (u_{1}u_2\cdots u_{p^1_1-1 })^\alpha 
e^{i\frac{\pi}4\sum\limits_{\ell=0}^{p^1_1-1}\si(T^\ell x)}
\\ 
\times \sum_{j=0}^{p^2_1-1} e^{i\frac{\pi}4\sum\limits_{\ell= p^1_1}^{p^1_1+j }\si(T^\ell x)}
 (u_{p^1_1}u_{p^1_1+1}\cdots u_{n_1+j})^\alpha \Omega( T^{p^1_1+j+1}(x)). 
\label{seuret:2}
\end{multline}
By the second item of Lemma~\ref{lem:22}, all the  $\sigma(T^{p^1_1+j+1}(x))$ are equal. 
We assume, without loss of generality, that they are equal to 1. 
Hence, we need  to take care of the sum  
\begin{equation}
\label{seuret:5}
\sum_{j=0}^{p^2_1-1} e^{i\frac{\pi}4 j}
(u_{p^1_1}u_{p^1_1+1}\cdots u_{p^1_1+j})^\alpha \Omega( T^{p^1_1+j+1}(x)).
\end{equation} 
Now, we use that the $T^{p^1_1+j+1}(x)$ are all close to 1. Since $\Omega$ 
is differentiable at 1, we have for every $j\in \{0,...,p^2_1-1\}$
$$
\Omega(T^{p^1_1+j+1}(x)) = \Omega(1) +  \Omega'(1) (T^{p^1_1+j+1}(x)-1) 
+ o(T^{p^1_1+j+1}(x)-1).
$$
Using this equation, we split the sum~\eqref{seuret:5}
into two parts:

\medskip

$(i)$ The sum 
$$
S_1=  \sum_{j=0}^{p^2_1-1} e^{i\frac{\pi}4 j}
 (u_{p^1_1}u_{p^1_1+1}\cdots u_{p^1_1+j})^\alpha \Omega(1)
$$
can be rewritten (since $e^{i\frac{\pi}{4}8}=1$)
$$ 
\Omega(1) \sum_{j=0}^{ \left\lfloor \frac{p^2_1-1}{8}\right\rfloor}  
(u_{p^1_1}u_{p^1_1+1}\cdots u_{p^1_1+8j-1})^\alpha  \left( \sum_{m=0}^{7} 
(u_{p^1_1+8 j  }u_{p^1_1+8 j +1}\cdots u_{p^1_1+8 j +m})^\alpha e^{i\frac{\pi} {4} m}\right)
$$ 
$$  +   \Omega(1) \sum_{ j = 8 \left\lfloor \frac{p^2_1-1} {8} \right\rfloor } ^{p^2_1-1}   
e^{i\frac{\pi}4 j}
(u_{p^1_1}u_{p^1_1+1}\cdots u_{p^1_1+j})^\alpha .
$$
The second sum  contains at most 7 terms all bounded in absolute 
value by 1, hence it is less than 7$\Omega (1)$.

We now consider the sum between parenthesis above. 
Fix $j\in \{0, ..., \left\lfloor \frac{p^2_1-1}{8}\right\rfloor\}$, 
and let us write for every   $m=0,..., 7,$ 
$$
 u_{p^1_1+8 j  }u_{p^1_1+8 j +1}\cdots u_{p^1_1+8 j +m} = 1-\ep^j_m.$$
Obviously from their definition, the $\ep^j_m$ are small positive quantities, 
and they are increasing with $m$. We thus have
\begin{align*}
\sum_{m=0}^{7} (u_{p^1_1+8 j  }u_{p^1_1+8 j +1}\cdots u_{p^1_1+8 j +m})^\alpha e^{i\frac{\pi} {4} m}  
& =   \sum_{m=0}^{7} (1-\ep^j_m)^\alpha e^{i\frac{\pi} {4} m} \\
& =  \sum_{m=0}^{7} (1- \alpha \ep^j_m + o(\ep^j_m)) e^{i\frac{\pi} {4} m} \\
& = - \sum_{m=0}^{7}   (\alpha\ep^j_m +o(\ep^j_m)) e^{i\frac{\pi} {4} m} .
\end{align*}
In absolute value, this sum is  less than $16 \ep^j_7$ 
(since the $\ep^j_m$ are small and increasing).  In addition, since
 the terms $u_{p^1_1+8 j +m}$ belong to the interval $(x_{16},1)$, one 
also has $8\ep^j_7 \leq C \ep_j^0 = C(1-u_{p^1_1+8 j  })$. Finally, we get 
\begin{equation}
\label{seuret:4}
|S_1|\leq  C\Omega(1) \left(1+ \sum_{j=0}^{ \left\lfloor \frac{p^2_1-1}{8}\right\rfloor}  
(u_{p^1_1}u_{p^1_1+1}\cdots u_{p^1_1+8j-1})^\alpha  (1-u_{p^1_1+8 j  })  \right).
\end{equation}
Using Lemma~\ref{lem:33} (in particular that $u_{p^1_1}\leq x_{p^1_2+16}$), 
we can bound by above this sum by
\begin{align*}
|S_1| &  \leq   C\Omega(1) \left(1+ \sum_{j=0}^{ \left\lfloor \frac{p^2_1-1}{8}\right\rfloor}  
\left ( \frac{p^2_1-8j}{p^2_1}\right)^\alpha  \frac{1}{p^2_1 -8j}\right)
\leq  C\Omega(1) \left(1+ \sum_{j=0}^{ \left\lfloor \frac{p^2_1-1}{8}\right\rfloor}  
 \frac{(p^2_1-8j)^{\alpha-1}}{ (p^2_1)^\alpha }\right)
\\
& =   C \Omega(1)
\end{align*}
for some constant $C$ independent of the problem.

\medskip

$(ii)$ 
The second  sum 
$$
S_2=  \sum_{j=0}^{p^2_1-1} e^{i\frac{\pi}4 j}
 (u_{p^1_1}u_{p^1_1+1}\cdots u_{p^1_1+j})^\alpha 
\Omega'(1) \Big(T^{p^1_1+j+1}(x)-1) +  o(T^{p^1_1+j+1}(x)-1)\Big)$$
is bounded in absolute value, by using  again  Lemma~\ref{lem:33} to 
find explicit upper bounds for the terms $u_{p^1_1}u_{p^1_1+1}\cdots u_{p^1_1+j}$
 and   $|T^{p^1_1+j+1}(x)-1|$.   We find, for some constant $C$ depending on $\Omega$ only, 
\begin{equation*}
|S_2|  \leq  C  \sum_{j=0}^{   p^2_1-1 }  
 \left ( \frac{p^2_1- j}{p^2_1}\right)^\alpha  \frac{1}{p^2_1 - j} 
\leq  C  \sum_{j=0}^{  p^2_1-1 }    \frac{(p^2_1- j)^{\alpha-1}}{ (p^2_1)^\alpha } =  C .
\end{equation*}
Hence, going back to~\eqref{seuret:2}, we obtain
$$
|\Sigma_{P_1-1} -\Sigma_{p^1_1-1} | \leq C  (u_{1}u_2\cdots u_{p^1_1-1 })^\alpha 
$$
for some constant depending on $\Omega$. The same holds if the 
$T^j(x)$ are all close to -1 and not to 1.

\medskip

\noindent {\bf Step 2:} We separate the cases where 
$0\leq j \leq P_1+p^1_2-1$ and $P_1+p^1_2 \leq j \leq P_1+p^1_2+p^2_2-1=P_2-1$.

\smallskip

$\bullet$ 
If $  0  \leq j \leq  p^1_2-1$:  all the terms $u_{P_1+j}$ satisfy $|u_{P_1+j}| \leq x_{16}$. 
We deduce that 
\begin{equation*}
|\Sigma_{P_1+p^1_2-1} -\Sigma_{P_1-1}|   \leq    
M_\Omega  (u_{1}u_2\cdots u_{P_1   })^\alpha  \sum_{j=0}^{ p^1_2-1} (x_{16})^{j} 
\leq  M_\Omega  \frac{ (u_{1}u_2\cdots u_{P_1   })^\alpha }{1-x_{16}}   .
\end{equation*}

\smallskip
$\bullet$ We now consider $ p^2_2\leq j  \leq P_2-1$.  
The same procedure as above (in Step 1) yields 
$$\left|\Sigma_{P_2-1} -\Sigma_{P_1+p^1_2-1} \right|\leq    
C (u_{1}u_2\cdots u_{P_1+p^1_2-1 })^\alpha   .$$

\medskip
\noindent {\bf Step ${\bf k}$:}
By an immediate recurrence, one obtains that for every $k\geq 1$, we have the following properties:
\begin{align*}
\tau_k^1:= \left| \Sigma_{P_k+p^1_{k+1}-1} -\Sigma_{P_k} \right|  & \leq   
M_\Omega  \frac{ (u_{1}u_2\cdots u_{P_{k}   })^\alpha }{1-x_{16}}  ,
\\
\tau_k^2:= \left| \Sigma_{P_{k+1}-1} -\Sigma_{P_k+p^1_{k+1}-1}  \right|& \leq    
C (u_{1}u_2\cdots u_{P_k+p^1_{k+1}-1 })^\alpha   .
\end{align*}

\medskip

We can now conclude regarding the convergence of the series~\eqref{seriesfinal}. 
Obviously by construction $p_k^2\geq 14$,  and  
the  first series $(\sum_{k\geq 1}  \tau_k^1)$ converges, 
since the ratio between two consecutive terms  
$\frac{(u_{1}u_2\cdots u_{P_{k+1}   })^\alpha}
{(u_{1}u_2\cdots u_{P_{k}   })^\alpha} $ is less than $(x_{16})^{14\alpha} <1$.

\medskip

For the second series $(\sum_{k\geq 1}  \tau_k^2)$, the same 
argument applies (the ratio between two consecutive terms  
$\frac{ (u_{1}u_2\cdots u_{P_{k+1}+p^1_{k+2}-1 })^\alpha  }
{   (u_{1}u_2\cdots u_{P_k+p^1_{k+1}-1 })^\alpha } $ is less than $(x_{16})^{14 \alpha} <1$).

 \subsection {Proof of Theorem~\ref{theo:5}, parts $(ii)$ and $(iii)$}

\

We consider the two  sums~\eqref{eq:seuret3} and~\eqref{eq:seuret4}. 
Using Theorem \ref{theo:5}$(i)$, the only problem may come from 
the  terms with $|T^j(x)| \leq 1/2$, since one can write 
$\log \frac{1}{|x|} = ( \log \frac{1}{|x|}) 
{\bf 1}_{[0,1/2)}(x) +\Omega (x)$, where $\Omega$ 
is bounded and differentiable at $1$ and $-1$. 
The same holds for  $\frac{1}{|x|^\beta}$. Thus we will 
focus on the convergence of the series
$$
\sum_{j=0}^{\infty}  \vert xT(x)\cdots T^{j-1}(x)\vert^\alpha 
\log \left(\frac{1}{T^j(x)} \right)  {\bf 1}_{[0,1/2)}(T^j(x)) 
$$
and 
$$
\sum_{j=0}^{\infty}   \frac
{\vert xT(x)\cdots T^{j-1}(x)\vert ^\alpha }{|T^j(x)|^\beta} {\bf 1}_{[0,1/2)}(T^j(x)). 
$$
We will prove the absolute convergence because the complex 
terms do not play any role any more in this convergence. The next lemma 
is key in the proof.

\begin{lem}
\label{lemmefinal}
If $|T^j(x)|\leq 1/2$, then there exists an integer $n$ such that
$$
|xT(x)\cdots T^{j-1}(x)|  \leq \frac{1}{Q_{n-1}} 
\ \mbox{ and } |T^j(x)|\geq \frac{1}{A_n }.
$$
Moreover, with two different $j$'s such that $|T^j(x)| \leq 1/2$ correspond two different $n$'s.
\end{lem}
\begin{proof}
This follows from a fine analysis  of the even and regular continued fractions. 
If $x$ is an irrational number whose RCF expansion $x=[A_1,A_2, .... ]_R$ 
contains only even numbers, then one knows that 
 the ECF expansion of $x$ is $x=[(1,A_1),(1,A_2), .... ]_E$, so 
$T^j(x) = G^j(x)$ for every $j\geq 1$, and the lemma follows 
from~\eqref{xGxG2x} and the definition of $G$.

We assume that this is not the case. Let us start with two observations on the shape of the 
mapping $T$:

\smallskip

 (R1)  $1/2< |y = T^{m-1}(x)|  <1$.  When $1/2< y <1$, 
then $T(y) $ is given by $T(y) =  2-\frac{1}{y}$. Similarly, 
when $-1<y<-1/2$, $T(y) =  -2-\frac{1}{y}$.  Subsequently, if 
$|T^{m-1}(x) | \geq 1/2$ for some $m$, then, recalling the 
algorithm~\eqref{singularization} producing the ECF from 
the RCF,   the $m$-term in the ECF expansion~\eqref{defevenx} 
of $x=[(e_1,a_1),(e_2,a_2),...]_E$ will be  $(-1,2) $. 

\smallskip
 (R2) $ y=T^{m-1}(x) \in (-1/2,1/2)$.  
 If  $T(y) \in (-1/2,0)$, then  
$\left\lfloor \frac{1}{ y}\right\rfloor $ is necessarily even. 
In this case, the $m$-th term in the expansion of $x$ is $(1,A_n)$ 
for some integer $n$. 
 
If $T(y) \in (-1/2,0)$, then $\left\lfloor \frac{1}{ y}\right\rfloor $  
is odd  and  $\left\lfloor \frac{1}{ T(y)}\right\rfloor $ is 
necessarily equal to $ 1$. This 
simply follows from the form of the mapping $T$. Again using 
the algorithm~\eqref{singularization},  
this amounts to change the RCF terms
$$ A_n  + \frac{1 }{1+\dfrac{1} {A_{n+2}+...}}$$ 
into
$$ A_n  + \frac{-1 }{ (A_{n+2}+1)+...},$$ 
in  order to get the ECF expansion of $x$, whose $m$-th term is 
necessarily $(-1,A_{n+2}+1)$ for some integer $n$.

We treat the case where $0<T^j(x)<1/2$, the other case is symmetric. 
There are two possibilities:

$\bullet$ If  $T^{j-1} (x) >1/2$: Necessarily,  we are at a step $j$ 
where, in the ECF expansion $x=[(e_1,a_1),(e_2,a_2),...  ]_E$ of $x$, 
there was a sequence of  $(-1,2)$ 
for some time before the index $j$. Then, this sequence 
of $(-1,2)$ stops at the $(j+1)$-th iterate, 
since for $y=T^j(x)$,   $T(y)$ is not 
 defined by $T(y) =  2-\frac{1}{y}$ any more.
 By our remark (R1) above,   this can be translated to the 
ECF and RCF expansions as follows: there is an integer $n$ such that
   $$ x = \frac{e_1 }{a_1+\dfrac{e_2}{a_2+\dfrac{e_3}
{\ddots +\dfrac{-1}{2+\dfrac{-1}{a_j+T^{j+1}(x)}}}}} =    
\frac{1}{A_1+\dfrac{1}{A_2+\dfrac{1}
{\ddots +\dfrac{1}{A_n  +G^n(x)}}}} $$ 
where $   a_j=A_n+1$.
 In particular,  this implies that $\displaystyle \frac{p_{j }}{q_{j }}
=\frac{P_n}{Q_n}$.  In this case, we know by Propositions~\ref{prop:ecf2} 
and~\ref{prop:ecf1} (as in Section~\ref{sec:10_1})  that 
$q_j-q_{j -1}\geq Q_{n-1}$. This yields  
$$
|xT(x)\cdots T^{j-1}(x)|   \leq  \frac{1}{q_{j +1}-q_{j}} 
\leq   \frac{1}{Q_{n-1}}.
$$
Moreover,  we have  $\frac{1}{T^{j}(x)}  -a_j \in (-1,1)$, 
with $a_j=A_n+1\geq 2$.  We deduce that  
$T^{j}(x) \geq \frac{1}{a_j-1 } = \frac{1}{A_n}.$

\medskip

$\bullet$
If  $ -1/2 < T^{j-1} (x)  < 1/2$: then  
one can apply our remark (R2) above:  for some integer $n$, one has  
$$ 
x = \frac{e_1 }{a_1+\dfrac{e_2}{a_2+\dfrac{e_3}
{\ddots +\dfrac{e_{\widetilde  j+m}}{ a_{\widetilde  j+m}
+\dfrac{ e_{\widetilde  j+m}}{a_{\widetilde  j+m}+T^{\widetilde  j+m+1}(x)}}}}} 
=    \frac{1}{A_1+\dfrac{1}{A_2+\dfrac{1}
{\ddots +\dfrac{1}{A_n  +G^n(x)}}}} 
$$  
with either $(e_{\widetilde  j+m},a_{\widetilde  j+m}) = (1,A_n)$ 
or $(e_{\widetilde  j+m},a_{\widetilde  j+m}) = (-1,A_n+1)$ for 
some integer $n$. In both case we have 
$\displaystyle \frac{p_{j }}{q_{j }}=\frac{P_n}{Q_n}$ for some 
integer $n$, and the same estimates as above hold true.
 
\medskip

It is obvious from the construction that the integers $n$ 
corresponding to different integers $j$ are pairwise distinct.
\end{proof}
  
Theorem~\ref{theo:5}$(ii)$-$(iii)$ follows immediately. Indeed, 
from Lemma \ref{lemmefinal}, we have
$$  
\sum_{j=0}^{\infty}  \vert xT(x)\cdots T^{j-1}(x)\vert^\alpha 
\log \left(\frac{1}{T^j(x)} \right)  {\bf 1}_{[0,1/2)}(T^j(x))  
\leq \sum_{n=1}^{+\infty} \frac{\log(A_{n+1})}{Q_{n}^\alpha}
\leq\sum_{n=1}^{+\infty} \frac{\log(Q_{n+1})}{Q_{n}^\alpha}
 $$
 and
 $$ 
\sum_{j=0}^{\infty}   \frac
 {\vert xT(x)\cdots T^{j-1}(x)\vert ^\alpha }{|T^j(x)|^\beta} {\bf 1}_{[0,1/2)}(T^j(x))  
\leq \sum_{n=1}^{+\infty}  \frac{A_{n+1}^{\beta}}{Q_{n }^\alpha}
\leq \sum_{n=1}^{+\infty} \frac{Q_{n+1}^\beta}{Q_{n}^{\alpha+\beta}},
$$
where we have used that $A_{n+1}\leq Q_{n+1}/Q_{n}$.
   
\section{Proof of Theorem \ref{theo:3}}\label{sec:10_1}

\subsection{Proof of Theorem \ref{theo:3}, parts $(i)$, $(ii)$ and $(iii)$} 
 
\

Let $\alpha>0$, $\beta\geq  0 $ be two positive real numbers.  
Let us rewrite the general term of the sum~\eqref{eq:riv10} as
$$
 \frac{\vert xT(x)\cdots T^{j-1}(x)\vert^{\alpha}}
{\vert T^j(x)\vert^{\beta}} =  \frac{\vert xT(x)\cdots T^{j-1}(x)\vert^{   \alpha+\beta}}
{\vert xT(x)\cdots  T^j(x)\vert^{\beta}}.
$$
Using~\eqref{majecf}, one sees that
$$\vert xT(x)\cdots  T^j(x)\vert \geq \frac{1}{2q_{j+1}} 
\ \ \mbox{ and } \ \ \vert xT(x)\cdots  T^{j-1}(x)\vert \leq \frac{1}{q_{j}-q_{j-1}} .$$

Hence
$$\sum_{j=1}^{\infty}  \frac{\vert xT(x)\cdots T^{j-1}(x)\vert^{\alpha}}
{\vert T^j(x)\vert^{\beta}} \leq  2^{ \beta }\sum_{j=1}^{\infty}  
\frac{  q_{j+1}^ {\beta}} { (q_{j}-q_{j-1})^{\alpha+\beta}}.
$$
We now use Proposition~\ref{prop:ecf2}:

\medskip

$\bullet$ If $q_j = Q_n$ for some integer $n$ and $q_{j-1} = Q_{n-1}$, then $q_{j+1}$ is either 
equal to $Q_{n+1}$ or to $Q_{n+1}+Q_n$. One also has 
$q_j-q_{j-1} =  Q_n-Q_{n-1} = ( A_n-1)Q_{n-1}+Q_{n-2}$. Since in this 
configuration, $A_n$ is necessarily even, we deduce that 
$q_j-q_{j-1}  \geq \frac{1}{2}A_nQ_{n-1} \geq \frac 1 4 Q_n$. Hence, 
$$ \frac{ q_{j+1}^{\beta}} { (q_{j}-q_{j-1})^{\alpha+\beta}} \leq 2^{2\alpha+3\beta}   
\frac{ Q_{n+1}^{\beta}} {Q_{n}^{\alpha+\beta}}.
$$

\medskip

$\bullet$ If $q_j = Q_n$ for some integer $n$ and $q_{j-1} = Q_n-Q_{n-1}$. 
 Again,  $q_{j+1}$ is either equal to $Q_{n+1}$ or to $Q_{n+1}+Q_n$.  
Hence,
$$ 
\frac{ q_{j+1}^{\beta}} { (q_{j}-q_{j-1})^{\alpha+\beta}} 
\leq 2^{\beta}  \frac{  Q_{n+1}^{\beta}} { Q_{n-1}^{\alpha+\beta}}.
$$

\medskip

$\bullet$ If $q_j =  mQ_n + Q_{n-1}$ for some integers $n\geq 1$ 
and $ 1\leq m \leq A_{n+1}-1$, then necessarily 
$q_j-q_{j-1}  \geq Q_{n }$, with equality when 
$m\geq 2$ (i.e. when  $q_{j-1} = (m-1) Q_{n } + Q_{n-1}$). 
Moreover, $q_{j+1}$ is  equal to $(m+1)Q_n + Q_{n-1} \leq 2(m+1) Q_n$. Hence, 
$$ \frac{ q_{j+1}^{\beta}} { (q_{j}-q_{j-1})^{\alpha+\beta}} 
\leq    \frac{  (2(m+1)Q_{n })^{\beta}} {Q_{n }^{\alpha+\beta}} 
=    2^\beta \frac{ (m+1)^{\beta}}{Q_n^{\alpha}}.$$

\medskip

We deduce from this analysis that for some constant 
 $C$ depending on $\alpha $ and $\beta$ only, 
\begin{align*}
\sum_{j=1}^{\infty}  \frac{\vert xT(x)\cdots T^{j-1}(x)\vert^{\alpha}}
{\vert T^j(x)\vert^{\beta}} &  \leq   C  \sum_{n=1}^{\infty}     
\frac{ Q_{n+1}^{\beta}} { Q_{n}^{\alpha+\beta}}  
+ C  \sum_{n=1}^{\infty}       \frac{ Q_{n+2}^{\beta}} { Q_{n}^{\alpha+\beta}} 
\\
& +  C \sum_{n=1}^{\infty}  
\sum_{m=1}^{A_{n+1}-1}   \frac{ (m+1)^{\beta}}{ Q_n^{\alpha}}.
\end{align*}
The first sum is clearly bounded by the second sum. For the third one, one sees that
\begin{equation*}
  \sum_{n=1}^{\infty}  \sum_{m=1}^{A_{n+1}-1}    
\frac{(m+1) ^{\beta}}{ Q_{n }^{\alpha} }\leq  C \sum_{n=1}^{+\infty}   
\frac{ A_{n+1} ^{\beta+1} }{ Q_{n }^{\alpha}} \leq 
C \sum_{n=1}^{\infty}      \frac{Q_{n+1}^{\beta+1 }}{ Q_{n }^{\alpha+\beta+1}},
\end{equation*}
where we used that $A_{n+1}\leq \frac{Q_{n+1}}{Q_n}$. 
Finally, 
\begin{equation*}
\sum_{j=1}^{\infty}  \frac{\vert xT(x)\cdots T^{j-1}(x)\vert^{\alpha}}
{\vert T^j(x)\vert^{\beta}} \leq  C  \sum_{n=1}^{\infty}      
 \frac{  Q_{n+2}^{\beta}} { Q_{n}^{\alpha+\beta}} +  
C  \sum_{n=1}^{\infty}  \frac{Q_{n+1}^{\beta+1 }}{Q_{n }^{\alpha+\beta+1}}.
\end{equation*}

Let us call $S_1$ and $S_2$ the two sums in the right-hand 
side above. Let $ \displaystyle \beta_\alpha= \frac{\sqrt{\alpha^2+4}-1}{2}$.  

\medskip

$\bullet$ $\beta <\beta_\alpha$. We rewrite the general term  of $S_2$ as
$$\frac{Q_{n+1}^{\beta+1 }}{Q_{n }^{\alpha+\beta+1}} 
= \left(\frac{ Q_{n+1} }{Q_{n }^{1+\frac{\alpha}{\beta+1}}}\right)^{ \beta+1}.$$
Then, observe that 
$$   
\frac{Q_{n+2}^{\beta} }{ Q_{n }^{\alpha+\beta}}  
=  \left(  \frac{Q_{n+2}}{ Q_{n+1 }^{1+\frac{\alpha}{\beta+1}}}\right)^{\beta}  
\left(  \frac{Q_{n+1}}{ Q_{n }^{1+\frac{\alpha}{\beta+1}}}\right)^{\beta(1+\frac{\alpha}{\beta+1})}  
Q_n^{\delta},
$$
where 
$$
\delta =  \beta\left(1+\frac{\alpha}{\beta+1}\right)^2 -(\alpha+\beta) 
= \frac{\alpha(\beta^2+\alpha\beta-1)}{(\beta+1)^2}.
$$
Because of our choice of $\beta$, $\delta<0$. Hence, the 
convergence of the series $S_1$ implies the convergence of $S_2$ 
(because the series $\sum_{n\geq 1} Q_n^\delta$ converges as 
soon as $\delta<0$ for all real numbers $x$).

One can also deduce that  the   series $S_1$ converges for all $x$ whose 
irrationality exponent $\mu(x) $ is smaller than $2+\frac{\alpha}{\beta+1}$, 
which implies the convergence of the series
$$ 
\sum_{n=1}^{\infty} \frac{Q_{n+2}}{ Q_{n+1 }^{1+\frac{\alpha}{\beta+1}}}.
$$

$\bullet$ $\beta >\beta_\alpha$. In this case, the same 
argument yields that   the convergence of the series $S_2$ 
implies the convergence of $S_1$. 

In terms of Diophantine properties, $S_2$ converges for all 
real numbers whose irrationality exponent $\mu(x)$  is smaller 
than  $1+\sqrt{1+\frac{\alpha}{\beta}} $. Indeed, for such 
an $x$, one has $\sum _{n\geq 1} \frac{Q_{n+1}}{Q_n^{\mu(x)-1}} <\infty$. Writing 
$$ \frac{Q_{n+2}^{\beta} }{ Q_{n }^{\alpha+\beta}}  
=  \left(  \frac{Q_{n+2}}{ (Q_{n+1 })^{\mu(x)-1}}\right)^{\beta}  
\left(  \frac{Q_{n+1}}{ Q_{n }^{\mu(x)-1}}\right)^{\beta(\mu(x)-1)}  
Q_n^{\beta(\mu(x)-1)^2 -\alpha-\beta},
$$
 we deduce that the series $S_2$ converges, since  
our choice for $\mu(x)$ implies $\beta(\mu(x)-1)^2 -\alpha-\beta<0$.

\medskip

$\bullet$ $\beta = \beta_\alpha$. The convergence of the two 
series are close to be equivalent, but they are not, 
depending on the values of $\alpha$ (and $\beta_\alpha$). 
It is simpler to indicate that for all real numbers 
$x$ with irrationality exponent less than 
$1+\sqrt{1+\frac{\alpha}{\beta_\alpha}} $ 
(which coincides with $2+\frac{\alpha}{\beta_\alpha+1}$), 
the two series $S_1$ and $S_2$ converge.

\subsection{Proof of Theorem \ref{theo:3}, part $(iv)$} 

\

The analysis is very similar to the one we performed in the last section, 
so we mention the main steps of the proof.  We write for an irrational $x\in (0,1)$
\begin{multline}
\vert xT(x)\cdots T^{j-1}(x)\vert ^\alpha\cdot \log\Big(\frac{1}{T^j(x)}\Big) =
 \vert xT(x)\cdots T^{j-1}(x)\vert ^\alpha\cdot\log \big (x T(x)...T^{j-1}(x)\big) \\
 -\vert xT(x)\cdots T^{j-1}(x)\vert ^\alpha\cdot \log \big(xT(x)...T^{j-1}(x)T^j(x)\big).
 \end{multline}
The first term in the right hand-side in dominated in absolute value by the 
second term, so we focus on the convergence of this term. Equation~\eqref{majecf} yields
$$
\vert xT(x) \cdots T^{j-1}(x)\vert^\alpha\cdot  \big|\log\big( xT(x)...T^{j-1}(x)T^j(x)\big)\big| 
\ll \frac{\log(q_{j+1})}{(q_j-q_{j-1})^\alpha}.
$$

The same distinction in three cases as in the previous section yields that
\begin{align*}
\sum_{j=1}^{\infty}  \vert xT(x)\cdots T^{j-1}(x)\vert^\alpha\cdot 
 \big|\log \big(xT(x)\cdots  T^j(x)\big) \big|  
& \ll &   \sum_{n=1}^{\infty}     \frac{   \log( Q_{n+1})} { Q_{n}^{\alpha }}  
+    \sum_{n=1}^{\infty}       \frac{ \log ( Q_{n+2}) } { Q_{n}^{\alpha }} 
\\
&& +    \sum_{n=1}^{\infty}  \sum_{m=1}^{A_{n+1}-1}   \frac{  \log (m+1)}{ Q_n^{\alpha}}.
\end{align*}
The second term dominates the first one.  Then we use that
$$  
\sum_{m=1}^{A_{n+1}-1}   \log (m+1) \ll A_{n+1} \log(A_{n+1})
$$
to prove that
\begin{equation*}
\sum_{j=1}^{\infty}  \vert xT(x)\cdots T^{j-1}(x)\vert^\alpha\cdot  
\big|\log \big(xT(x)\cdots  T^j(x)\big) \big|  \ll      \sum_{n=2}^{\infty} 
\frac{ \log ( Q_{n+1}) } { Q_{n-1}^{\alpha }} +    \sum_{n=1}^{\infty}   
 \frac{  A_{n+1} \log(A_{n+1})}{ Q_n^{\alpha}}.
\end{equation*}
Using that $A_{n+1}\leq \frac{Q_{n+1}}{Q_n}$, we finally get~\eqref{condlog}.

\def\refname{Bibliography}

\end{document}